\documentclass[11 pt,a4paper,Rectoverso]{article}  

\usepackage{hyperref}
\usepackage{times}

\hypersetup{
    bookmarks=true,         
    unicode=false,          
    pdftoolbar=true,        
    pdfmenubar=true,        
    pdffitwindow=true,      
    pdfnewwindow=true,      
    colorlinks=true,       
    linkcolor=blue,          
    citecolor=blue,        
    filecolor=blue,      
    urlcolor=blue           
}
\usepackage[latin1]{inputenc}

\usepackage[numbers,sort&compress]{natbib}
\usepackage{amsmath,indentfirst,qsymbols,graphicx,psfrag,amsfonts,amssymb,stmaryrd}

\usepackage{amsthm}

\newcommand{\event}[1]{\left[\, #1 \,\right]}

\newcommand{\E}[1]{\mathbb{E}\event{#1}}
\newcommand{\Ec}[1]{\mathbb{E}[#1]}

\newcommand{\bc}{{\mathbf c}}

\DeclareMathOperator{\var}{var}

\def \nk{{\sf n}_\kappa}

\def \ben{\begin{eqnarray}}
\def \een{\end{eqnarray}}
\def \be{\begin{eqnarray*}}
\def \be{\begin{eqnarray*}}
\def \bq{\begin{equation}}
\def \eq{\end{equation}}
\def \ee{\end{eqnarray*}}

\newcommand{\eqdist}{\ensuremath{\stackrel{d}{=}}}

\def \ee{\end{eqnarray*}}
\def \1{{\bf 1}}

\def \Cont{{\sf Cont}}

\def \Forests{{\mathbb F}}
\def \F{{\sf F}}
\def \I{{\bf 1}}
\def \LL{\L ukasiewicz }
\def \LR{{\sf LR}}

\def \R{{\sf R}}

\def \Trees{\mathbb T}

\def \bM{{\bf M}}

\def \bm{{\bf m}}
\def \bp{{\bf p}}
\def \bp{{\bf p}}
\def \bs{{\bf s}}
\def \bt{{\bf t}}
\def \build#1#2#3{\mathrel{\mathop{\kern 0pt#1}\limits_{#2}^{#3}}}
\def \soussur#1#2#3{\mathrel{\mathop{\kern 0pt#1}\limits_{#2}^{#3}}}
\def \sursous#1#2#3{\mathrel{\mathop{\kern 0pt#1}\limits^{#2}_{#3}}}

\def \bu{{\bf u}}

\def \cro#1{\llbracket#1\rrbracket}

\def \dd{\xrightarrow[n]{(d)}}

\def \eref#1{(\ref{#1})}
\def \floor#1{\lfloor#1\rfloor}
\def \imp{\Rightarrow}

\def \l{\left}

\def \r{\right}
\def \se{{\sf e}}
\def \sb{{\sf b}}

\def \sous#1#2{\mathrel{\mathop{\kern 0pt#1}\limits_{#2}}}
\def \sur#1#2{\mathrel{\mathop{\kern 0pt#1}\limits^{#2}}}

\def \st2#1#2{\left\{\!\!\!
	\begin{array}{c}#1\\#2
        \end{array}\!\!\!
	\right\}}

\newqsymbol{`B}{\mathcal{B}}
\newqsymbol{`E}{\mathbb{E}}
\newqsymbol{`I}{\mathbb{I}}
\newqsymbol{`N}{\mathbb{N}}
\newqsymbol{`O}{\Omega}
\newqsymbol{`P}{\mathbb{P}}
\newqsymbol{`Q}{\mathbb{Q}}
\newqsymbol{`R}{\mathbb{R}}
\newqsymbol{`W}{\mathbb{W}}
\newqsymbol{`Z}{\mathbb{Z}}
\newqsymbol{`a}{\alpha}
\newqsymbol{`e}{\varepsilon}
\newqsymbol{`o}{\omega}
\newqsymbol{`t}{\tau}
\newqsymbol{`w}{{\cal W}}
\newcommand{\bC}{{\bf{C}}}

\title{\bf  Asymptotics of trees with a prescribed degree sequence and applications}
\author{Nicolas Broutin\thanks{Projet Algorithms, Inria Rocquencourt, Domaine de Voluceau, 78153 Le Chesnay - France. Email: nicolas.broutin@inria.fr. Partially supported by the grant ANR-09-BLAN-0011 Boole.} \and Jean-Fran\c{c}ois Marckert\thanks{LaBRI, Universit\'e de Bordeaux - CNRS, 351 cours de la Lib\'eration, 33405 Talence - France. Email: marckert@labri.fr. Partially supported by the grant ANR-08-BLAN-0190-04 A3.}}

\begin{document}

\renewcommand{\baselinestretch}{1.2}
\newtheorem{exe}{Exercise}
\newtheorem{fig}{\hspace{1cm} Figure}
\newtheorem{lem}{Lemma}
\newtheorem{conj}{Conjecture}
\newtheorem{defi}{Definition}
\newtheorem{pro}[lem]{Proposition}
\newtheorem{theo}[lem]{Theorem}
\newtheorem{cor}[lem]{Corollary}
\newtheorem{remi}{Remark\rm}{\rm}
\newtheorem{com}{Comments\rm}{\rm}
\newenvironment{rem}%
{\begin{center}\begin{minipage}{16cm}\begin{remi}}%
{\end{remi}\end{minipage}\end{center}}
\newcounter{aaa}
\def \ds{{\bf s}}
\def \wh#1{\widehat{#1}}
\newtheorem{note}{Note \rm}{\rm}

\maketitle

\begin{abstract}Let $t$ be a rooted tree and $n_i(t)$ the number of nodes in $t$ having $i$ children. The degree sequence $(n_i(t),i\geq 0)$ of $t$ satisfies  $\sum_{i\ge 0} n_i(t)=1+\sum_{i\ge 0} in_i(t)=|t|$, where $|t|$ denotes the number of nodes in $t$.
In this paper, we consider trees sampled uniformly among all { plane} trees having the same degree sequence $\ds$; we write $`P_\ds$ for the corresponding distribution. Let $\ds(\kappa)=(n_i(\kappa),i\geq 0)$ be a list of degree sequences indexed by $\kappa$ corresponding to trees with size $\nk\to+\infty$. We show that under some simple and natural hypotheses on $(\ds(\kappa),\kappa>0)$ the trees sampled under $`P_{\ds(\kappa)}$ converge to the Brownian continuum random tree after normalisation by $\nk^{1/2}$. Some applications concerning Galton--Watson trees and coalescence processes are provided.  
\end{abstract}

\section{Introduction}\label{sec:intro}

 Let $t$ be a rooted tree and $n_i(t)$ the number of nodes in $t$ having $i$ children. The sequence $(n_i(t),i\geq 0)$ is called the degree sequence of $t$, and satisfies  $ \sum_{i\ge 0} n_i(t)=1+\sum_{i\ge 0} in_i(t)=|t|$, the number of nodes in $t$. 

The aim of this paper is to study trees chosen under $`P_{\ds}$, the uniform distribution on the set of { plane} trees with specified degree sequence $\ds=(n_i,i\geq 0)$, and then size $|\ds|:=\sum_{i\ge 0} n_i$. More precisely, a sequence of degree sequences $(\ds(\kappa),\kappa \geq 0)$ with $\ds(\kappa)=(n_i(\kappa),i\geq 0)$, corresponding to trees with size $\nk:=|\ds(\kappa)|\to +\infty$ is given, and the investigations concern the limiting behaviour of tree under $`P_{\ds(\kappa)}$. 
\begin{figure}[htb]
\centerline{\includegraphics[height=1.6cm]{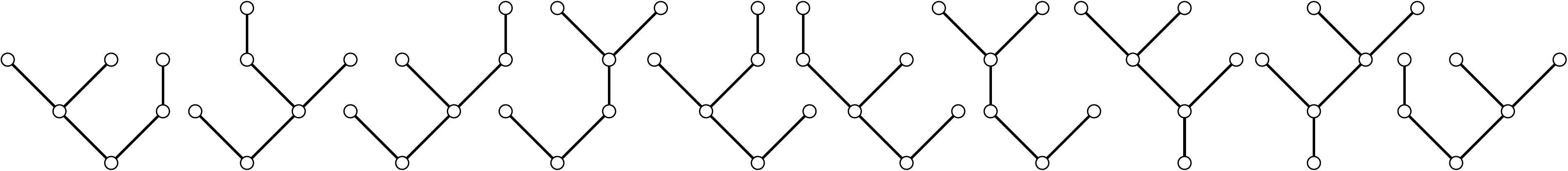}}
\caption{The 10 trees of $\Trees_\ds$ for the degree sequence $\ds=(3,1,2,0,0,\dots)$.}
\end{figure}

We now introduce some notation valid { in the entire paper}. We denote by $\bp(\kappa)=(p_i(\kappa),i\geq 0)$ the degree distribution under $`P_{\ds(\kappa)}$:
\bq\label{eq:pik}
p_i(\kappa)=\frac{n_i(\kappa)}{\nk}.
\eq
{ Let also
\bq\label{eq:sigma}
\sigma_\kappa^2:=\sum_{i\ge 1} \frac{n_i(\kappa)}{\nk-1} i^2 - 1;
\eq 
$\sigma_\kappa^2$ is ``almost'' the associated variance, this choice of definition yields shorter formulae in the following.} The maximum degree of any tree with degree sequence $\ds(\kappa)$ is 
\[\Delta_\kappa=\max\{i:n_i(\kappa)>0\}.\]
{ Throughout} the paper $\bp=(p_i,i\geq 0)$ is a distribution with mean 1, and  variance $\sigma^2_\bp\in(0,+\infty)=\sum_{i\ge 0} i^2p_i -1\in (0,\infty).$ In the following theorem, which is the main result of the present paper,  $\bp(\kappa)\imp \bp$ means equivalence in distribution, which here means that for any $i\geq 0$, $p_i(\kappa)\to p_i$, as $\kappa\to \infty$.
\begin{theo}\label{thm:main_gh} 
Let $(\ds(\kappa),\kappa\geq 0)$ be a sequence of degree sequences such that $\nk\to +\infty$, $\Delta_\kappa=o(\nk^{1/2})$,  $\bp(\kappa)\imp \bp$ with  $\sigma_\kappa^2\to \sigma^2_\bp$, that is convergence of second moment. Let $\bt$ be a plane tree chosen under $`P_{\ds(\kappa)}$ and let $d_{\bt}$ be the graph distance in $\bt$. Under $`P_{\ds(\kappa)}$, when $\kappa\to +\infty$, $(\bt,\sigma_\kappa \nk^{-1/2} d_{\bt})$ converges in distribution to Aldous' continuum random tree (encoded by twice a Brownian excursion), in the Gromov--Hausdorff sense.
\end{theo}

First observe that the very strong result of \citet{HaMi2010} about the asymptotics of Markov branching trees that has been used to give asymptotics for random trees in a wide variety of settings does not apply in the present case of trees with a prescribed degree sequence. Indeed, the subtrees of a given node are not independent given their sizes when one fixes the degree sequence. Our approach uses instead the observation done by \citet{MA-MO} that all natural encodings of the trees are asymptotically proportional in the case of Galton-Watson trees conditioned by the size. The same property will also hold here. In particular, the \emph{height process} or the \emph{contour process} both encoding the metric structure of the tree resemble the \emph{depth-first queue process} encoding the sequence of degrees observed when performing a depth-first traversal. This fact was used by \citet{MA-MO} to give an alternative proof of Aldous' result in the case of Galton--Watson trees conditioned on the total progeny under some moment condition (\citet{BK} also observed this phenomenon). 

\medskip
 One of the crucial questions underlying our work is that of the universality of the convergence of random trees to the continuum random tree (CRT).

We are motivated by the metric structure of graphs with a prescribed degree sequence. Introduced by \citet{BeCa1978a} and by \citet{Bo1980a} in the form of the configuration model, these graphs have received a lot of attention since the first tight analysis of the size of connected components by \citet{MoRe1995,MoRe1998}. This is mainly because the model allows for a lot of flexibility in the degree sequence. In particular, the model provides a construction of random graphs with degree sequences that may match the observations in large real-world networks.

Of course, random graphs with a prescribed degree sequence are much more complex than trees with a prescribed degree sequence, but there is no doubt that the analysis of trees is a first step towards the identification of the metric structure of the corresponding graphs. Indeed, recent results of \citet{Joseph2010a} show that under some moment condition, the sizes of the connected components of  random graphs with a prescribed \emph{critical} degree sequence are similar to those of Erd\H{o}s--R\'enyi $G(n,p)$ random graphs \cite{ErRe1960,Bollobas2001,JaLuRu2000}: they may be asymptotically described in terms of the lengths of the excursions of a Brownian motion with parabolic drift above its current minimum, as demonstrated by \citet{Aldous1997}. (See also \cite{Riordan2011a}, where it is supposed that the maximum degree is bounded.) On the other hand, the metric structure of $G(n,p)$ inside the critical window has recently been identified in terms of modifications of Brownian CRT by \citet*{AdBrGo2010,AdBrGo2010a}. In other words, the present analysis is one more building block towards an invariance principle for scaling limits of random graphs, i.e., that critical random graphs with a prescribed degree sequence have (under a suitable moment condition on the degree distribution) the same scaling limit (as sequence of compact metric spaces) as classical random graphs \cite{AdBrGo2010a}. This is at least what is suggested by the results of \citet*{Hofstad2009,BhHoLe2009}, \citet{Joseph2010a} and \citet{Riordan2011a}. \par

	Moreover, in the same way that uniform random trees or forests may be seen as the results of coagulation/fragmentation processes involving particles \cite{Pitman1999b,Pitman2006}, trees with a prescribed degree sequence appear naturally in similar aggregation processes. The model where particles have constrained valence may appear more ``physically'' grounded. The relevant underlying coalescing procedure is the additive coalescent \cite{AlPi1998a, Bertoin2000a}, a Markov process whose dynamics are such that particles merge at a rate proportional to the sum of their masses/sizes. The additive coalescent is the aggregation process appearing in Knuth's modification of R\'enyi's parking problem \cite{Renyi1958,Hemmer1989} or the hashing with linear probing \cite{ChLo2002,BeMi2006}. The reader may find more information about coagulation/fragmentation processes in the monograph by \citet{Bertoin2006} or the recent survey by \citet{Berestycki2009}. \medskip

The model $`P_{\ds}$ is related to Galton--Watson trees \cite{AtNe1972, Harris1963}, also called simply generated trees in the combinatorial literature, by a simple conditioning: the distribution $`P_{\ds}$ coincides with the distribution of the family tree $\bt$ of a Galton--Watson process with offspring distribution $(\nu_i,i\geq 0)$ (which must satisfies $\nu_i>0$ if $n_i>0$) conditioned on $\{n_i(\bt)=n_i,i\geq 0\}$. { Indeed, $`P_{\ds}$ assigns the same probability to all trees with the same degree sequence.}
In this sense, the distribution $\nu$ plays a role of secondary importance, and $`P_{\ds}$ appears to be a model of combinatorial nature, { far from the world} of Galton--Watson processes.
Nevertheless, we will see that Theorem \ref{thm:main_gh} implies the following result of Aldous (stated in a slightly different form in \cite{Aldous1991}) (see also  \cite{Aldous1991,Aldous1991b,Aldous1993a, MA-MO, Legall1993}), { where $H_\bt$ is the height process of $\bt$ (the definition is recalled in the next section)}.
\begin{pro}[Aldous \cite{Aldous1991}]\label{pro:Ald} Let $\mu=(\mu_i,i\geq 0)$ be a distribution  with mean  $m_\mu=1$ and variance $\sigma^2_\mu\in(0,+\infty)$, and let $`P_{\mu}$ be the distribution of a Galton--Watson tree with offspring distribution $\mu$. Along the subsequence $\{n~: `P_\mu(|\bt|=n)>0\}$, under $`P_\mu(~\cdot~|\,|\bt|=n)$ 
\[\l(\frac{H_\bt(nx)}{\sqrt{n}}\r)_{x\in[0,1]}\xrightarrow[n\to\infty]{(law)} \frac{2}{\sigma_\mu} \se\]
where $\se$ denotes a standard Brownian excursion, the convergence holding in the space ${\mathcal C}[0,1]$ equipped with the topology of uniform convergence. 
\end{pro}
We will see that this theorem may be seen indeed as a consequence of Theorem \ref{thm:main_gh}; the argument morally relies  on the fact that under $`P_\mu(~.~|\,|\bt|=n)$, the empirical degree sequence satisfies the { hypotheses} of Theorem \ref{thm:main_gh} with probability going to 1 (this is stated in Lemma \ref{lem:GW}). The proof of this theorem is postponed { until} Section \ref{sec:pt}. 

{ Note also results of \citet{Rizzolo2011a} and \citet{Kort} that have a flavor similar to our Theorem~\ref{thm:main_gh} (although neither implies the other): they proved that Galton--Watson trees conditioned on the number of nodes having their degrees in a subset $A$ of the support of the measure $\mu$ has a limiting behaviour depending on $A$. For instance, they consider trees conditioned on the number of leaves, the number of nodes with other out-degrees being left free. The proofs in \citet{Rizzolo2011a} rely ultimately on the approach based on Markov branching trees developed by \citet{HaMi2010}.
}

\medskip
\noindent\textsc{Plan of the paper.}\ In Section~\ref{sec:tree} we introduce precisely the model of trees we consider. Section~\ref{sec:backbone} is devoted to a useful backbone decomposition for these trees. We then prove our main result, the convergence of rescaled trees to the continuum random trees, in Section~\ref{sec:proof}. Finally, the application to coagulation processes with particles with constrained valence is developed in Section~\ref{sec:coagulation}.

\section{Trees with prescribed degree sequence}
\label{sec:tree}
We here define formally the combinatorial object discussed in this paper.
For convenience we write $\mathbb{N}=\{1,2,\dots\}$ for the set of positive natural numbers. First recall some definitions related to standard rooted plane trees. Let $\mathcal U=\bigcup_{n\geq 0}\mathbb{N}^n$
be the set of finite words on the alphabet $\mathbb{N}$, where $\mathbb N^0=\{\varnothing\}$, and $\varnothing$ denotes the empty word. Denote by $uv$  the concatenation of $u$ and $v$; by convention ${\varnothing} u=u{\varnothing} =u$. 

A subset $T$ of $\mathcal U$ is a \it plane tree \rm (see Figure \ref{fig:pt}) if 
{ \begin{itemize}
	\item it contains $\varnothing$ (called the root),
	\item it is stable by prefix (if $uv\in T$ for $u$ and $v$ in $\mathcal U$, then $u\in T$), and 
	\item if ($uk\in T$ for some $k>1$ and $u\in U$) then $uj\in T$ for $j$ in $\{1,\dots,k\}$.
\end{itemize} }
This last condition appears necessary to get a unique tree with a given genealogical structure.
The set of plane trees will be denoted by $\Trees$.

\begin{figure}[ht]
\psfrag{v}{$\varnothing$}
\psfrag{1}{1}
\psfrag{11}{11}
\psfrag{12}{12}
\psfrag{13}{13}
\psfrag{14}{14}
\psfrag{15}{15}
\psfrag{131}{131}
\psfrag{151}{151}
\psfrag{152}{152}
\centerline{\includegraphics[height=3cm]{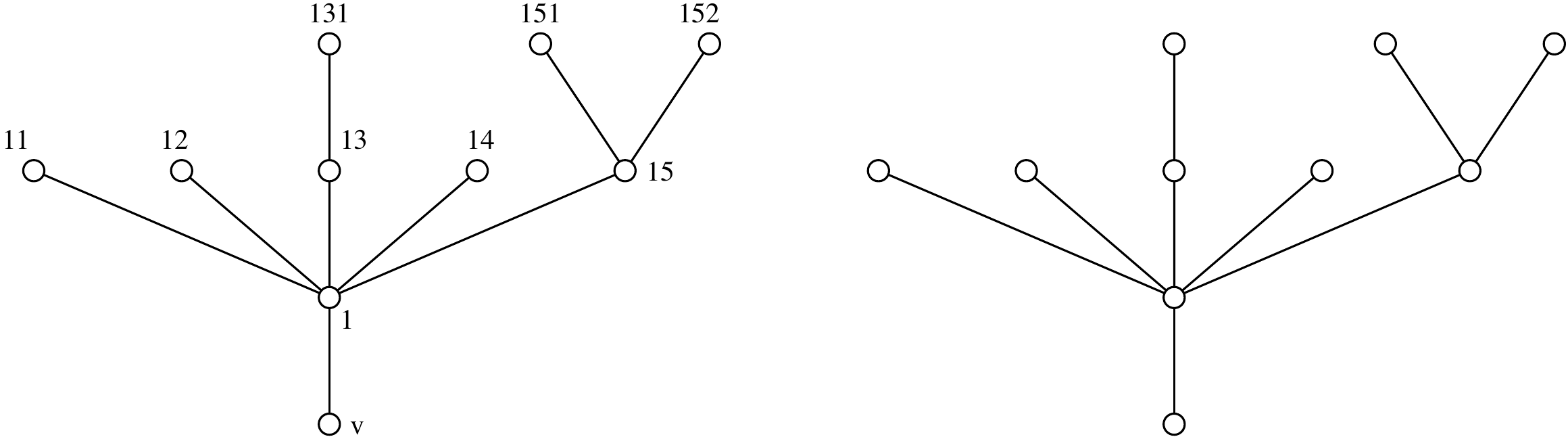}}
\caption{\label{fig:pt} Usual representation of the plane tree  $\{\varnothing,1,11,12,13,14,15,131,151,152\}$}
\end{figure}

Notice that the lexicographical order $<$ on $\mathcal U$, also named the depth-first order,  induces a total order on any tree $t$; this is of prime importance for the encodings of $t$ we will present.
For $t\in \Trees$, and $u\in t$, let $c_t(u)=\max\{i~: ui \in t\}$ be the number of children of $u$ in  $t$. The depth of $u$ in $t$, its number of letters as a word in $\mathcal U$, is denoted $|u|$. The notation $|t|$ refers to the cardinality of $t$, its number of nodes including the root $\varnothing$.
 
With a tree $t\in \Trees$, one can associate its degree sequence $\ds(t)=(n_i(t), i\ge 0)$, where $n_i(t)=\#\{u\in t: c_t(u)=i\}$ is the number of nodes with degree $i$ in $t$. For a fixed degree sequence $\ds$, write $\Trees_\ds$ for the set of trees $t\in \Trees$ such that $\ds(t)=\ds$, and let $`P_\ds$ be the uniform distribution on $\Trees_\ds$. To investigate the shape of random trees under $`P_\ds$, we will use the usual encodings: \emph{height process} $H$  and \emph{depth-first walk} $S$ (or \LL path) { and \emph{contour process} $\bC$}. These encodings are defined by first fixing their values at the integral points, and then linear interpolation in between (See Figure~\ref{fig:thw}). For a tree $t\in \Trees$, let $\tilde u_1=\varnothing< \tilde u_2<\dots<\tilde u_{|t|}$ denote the nodes of $t$ sorted according to the lexicographic order. Then we define $H=H_t$ by $H(i)=|\tilde u_{i+1}|$, $S=S_t$ by $S_t(i)=\sum_{j=1}^i (c_t(\tilde u_j)-1)$; the process $H_t$ is defined on $[0,|t|-1]$ and $S_t$ on $[0,|t|]$. 
{ For the contour process $\bC_t$ of $t$, we need to define first a function $f_t:\{0,\dots,2(|t|-1)\}\mapsto t$ which can be regarded as a walk around $t$; first set $f_t(0)= \varnothing$, the root. For $i<2(|t|-1)$, given $f_t(i)=v$, $f_t(i+1) $ is $u$, the smallest child of $v$ (for the lexicographical order) absent from the list $\{f_t(0),\dots,f_t(i)\}$ , and the father of $v$ if no such $u$ exists. The contour process has the following values on integer positions 
\[\bC_t(i)=|f_t(i)|,~~i\in\{0,\dots,2(|t|-1)\}.\]}
\begin{figure}[ht]
\centerline{\includegraphics[height=3cm]{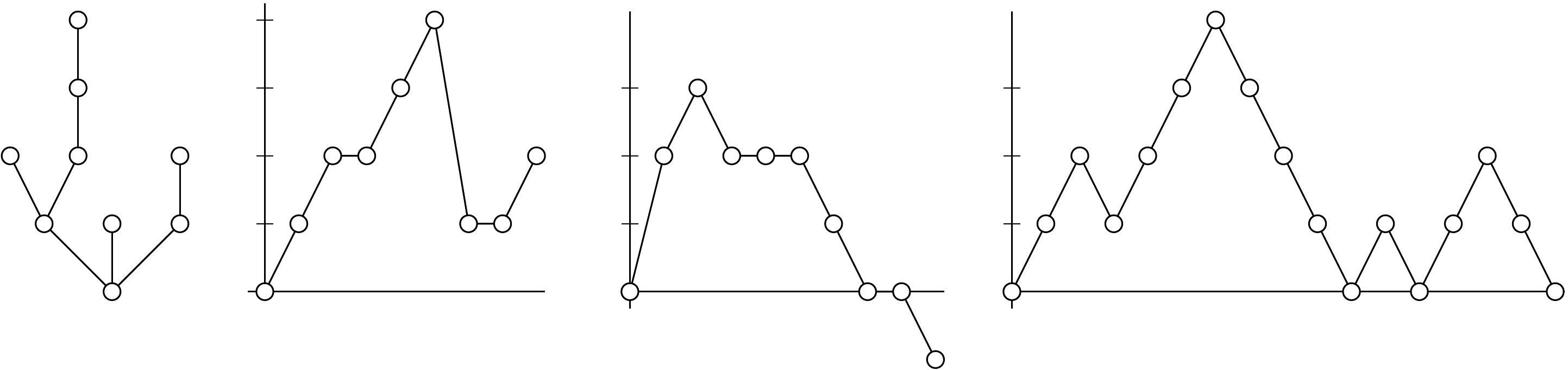}}
\caption{\label{fig:thw} A plane tree $t\in \Trees$, its height process $H_t$, \LL walk $S_t$ and its contour process $\bC_t$.}
\end{figure}

\begin{theo}\label{thm:main_encodings}Under the hypothesis of Theorem \ref{thm:main_gh}, under $`P_{\ds(\kappa)}$,
\begin{equation}\label{eq:all-conv}
\l(\frac{H_\bt( x (\nk-1))}{\nk^{1/2}},{ \frac{C_\bt( x 2(\nk-1))}{\nk^{1/2}}}, \frac{S_\bt( x   \nk )}{\nk^{1/2}}\r)_{x\in [0,1]}\xrightarrow[\kappa\to\infty]{} \l(\frac 2 {\sigma_\bp}\,\se,\frac 2 {\sigma_\bp}\,\se,\sigma_\bp\se\r)
\end{equation}
in distribution in the space $\mathcal C([0,1],`R^3)$ of continuous functions from $[0,1]$ with values in $`R^3$, equipped with the supremum distance.
\end{theo}{ The contour process is a kind of interpolation of the height process. The fact that both these processes have the same asymptotic behaviour is well understood in some general settings : it is shown in Marckert and Mokkadem (Lemma 3.19 \cite{MAMO2}) that, if under any model of random trees, the height process has a continuous limit after a non trivial normalisation, then the contour process has the same limit with the same space normalisation (and time normalisation multiplied by 2 to take into account the relative durations of these processes).  This property has been noticed before in the case of Galton--Watson trees conditioned by the size \cite{BK,MA-MO}.\par
As a consequence (of Lemma 3.19 \cite{MAMO2}), to establish 
\begin{equation}\label{eq:all-conv2}
\l(\frac{H_\bt( x (\nk-1))}{\nk^{1/2}}, \frac{S_\bt( x   \nk )}{\nk^{1/2}}\r)_{x\in [0,1]}\xrightarrow[\kappa\to\infty]{} \l(\frac 2 {\sigma_\bp}\,\se,\sigma_\bp\se\r)
\end{equation}
is sufficient to deduce \eref{eq:all-conv}.  }

Note now that the condition $\sigma^2_\bp>0$ is necessary in Theorem \ref{thm:main_encodings}: it ensures that $p_0=\lim_{\kappa\to\infty} n_0(\kappa)/\nk>0$ and that large trees are not close to a linear tree, where most of the nodes have degree one.\par
 A tree $t\in \Trees$ can also be seen as a metric space when equipped with the graph distance $d_t$. A consequence of Theorem~\ref{thm:main_encodings} is that, under $`P_{\ds(\kappa)}$, the metric space $$\l(\bt,\frac{\sigma_\kappa}{\sqrt{\nk}}d_\bt\r)$$ converges to the continuum random tree encoded by $2\se$ in the sense of Gromov--Hausdorff distance between equivalence classes of compact metric spaces. The fact that the convergence { of the contour process (or the height process)} implies the convergence of the trees for the Gromov--Hausdorff topology is well known, see for example Lemma 2.3 in Le Gall \cite{LGRRT}. So, in particular, to prove Theorem~\ref{thm:main_gh} it suffices to prove Theorem~\ref{thm:main_encodings} and for this,  it is sufficient to prove \eref{eq:all-conv2}.

\medskip
\noindent\textbf{Remark.}\ One can define other models of random trees with a prescribed degree sequence: for example, \emph{rooted labelled trees}. Let $`Q_{\ds(k)}$ be the uniform distribution on those with degree sequence $\ds(k)$. Since labelled trees have a canonical ordering (using an order on the labels to order the children of each node), forgetting the labels, they can be seen as plane trees with the same degree sequence, inducing a distribution $`P'_{\ds(k)}$ on the set of plane trees. By a simple counting argument, it turns out that $`P'_{\ds(k)}=`P_{\ds(k)}$. This situation is drastically different from the general case, since the projection of uniform labelled trees on plane tree (that is without fixing the degree sequence) does not induce the uniform distribution on plane trees. As a consequence, Theorem~\ref{thm:main_gh} is also valid for the model of labelled trees with a prescribed degree sequence.

\section{Combinatorial considerations: a backbone decomposition}\label{sec:backbone}

In this section we develop a decomposition of trees under $`P_{\ds(k)}$ along a branch. It is essentially the usual \emph{backbone decomposition} for Galton--Watson trees due to \citet*[see, e.g.,][]{LyPePe95a} transposed under $`P_{\ds(k)}$. The decomposition amounts to describing the structure of the branch from the root to a distinguished node $u$, together with the (ordered) forest formed by the trees rooted at the neighbours of that branch. \medskip

\noindent\textsc{Forest with a given degree sequence.}\
A forest ${\sf f}=(t_1,\dots,t_k)$ is a finite sequence of trees; its degree sequence $\ds({\sf f})=\sum_{i=1}^k \ds(t_i)$ is the (component-wise) sum of the degree sequences of the trees which compose it.
If $\ds=(n_i,i\geq 0)$ is the degree sequence of a forest ${\sf f}$, then the number of roots of ${\sf f}$ is given by  $r=|\ds|-\sum_{i\ge 0} i n_i$. Let $\Forests_\ds$ be the set of forests of ($r$ ordered) plane trees having degree sequence $\ds$.

We have (see, e.g., \cite{Pitman1999b}, p. 128) 
\bq\label{eq:nbf}
\#\Forests_\ds=\frac{r}{|\ds|}\binom{|\ds|}{(n_i,i\geq 0)}=\frac r {|\ds|} \cdot \frac{|\ds| !}{\prod_{i\ge 0} n_i!}.
\eq

\medskip
\noindent\textsc{The content of a branch.}\ Let $t$ be a plane tree, and let $u=i_1\dots i_{|u|}$ be one of its nodes, where $i_j\in \mathbb{N}$ for any $j$. 
For  $j\leq |u|$, write $u_j=i_1\dots i_j$, the ancestor of $u$ having depth $j$ (with the convention $u_0=\varnothing$, the root of $t$). The set $\cro{\varnothing,u}=\{u_j~: j <|u|\}$ is called the branch of $u$ (notice that $u$ is excluded).
For any $i\geq 0$, the number of ancestors of $u$ having $i$ children is written
\[M_i(u,t)=\#\{v~:v\textrm{ strict ancestors of }u,c_t(v)=i\}.\]
We refer to $\bM(u,t)=(M_i(u,t), i\ge 0)$ as the composition of the branch. Note that we necessarily have $M_0(u,t)=0$.
Clearly if $u\in t$, then
\bq |u|=\sum_{i\geq 1}M_i(u,t)=|\bM(u,t)|.\eq
Further let $\LR(u,t)$ (for left or right) be { the set of nodes that are children of some node in $\cro{\varnothing, u}$ without being themselves in $\cro{\varnothing, u}$; note that because of our convention for $\cro{\varnothing,u}$, $u$ belongs to $\LR(u,t)$ (see Figure~\ref{fig:RandLR})}.
\begin{figure}[h]\psfrag{u}{$u$}
\centerline{\includegraphics[height=4.5cm]{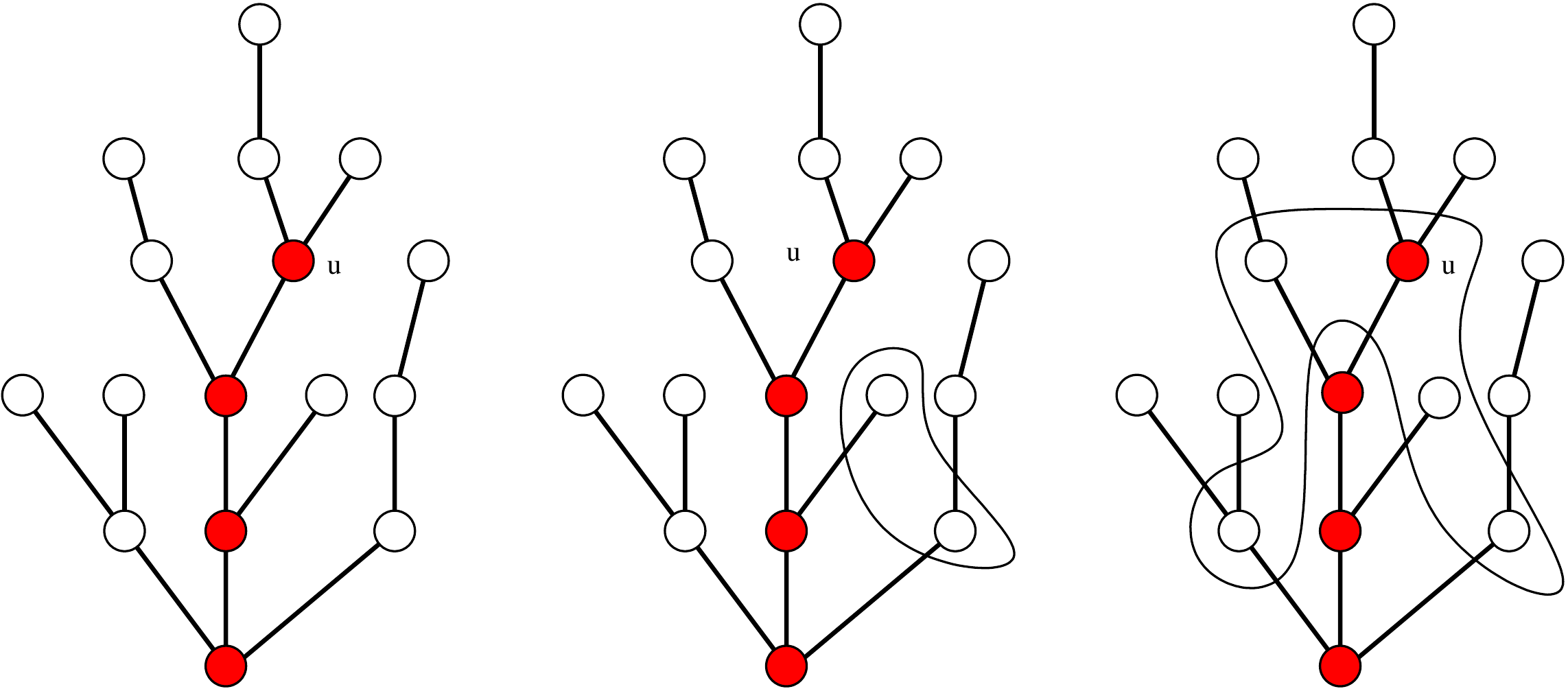}}
\caption{\label{fig:RandLR} { A tree $t$ with a marked node $u$; the sets in the two right-hand side pictures show the sets  $\R(u,t)$ and $\LR(u,t)$.}}
\end{figure}
Let also $\R(u,t)$ be the subset of $\LR(u,t)$, of nodes lying to the right of the path  $\cro{\varnothing,u}$ { (therefore $u\notin \R(u,t)$)}.
A node $v$ is in $\R(u,t)$ if it is a child of some $u_i$, for $i\in \{0,\dots, |u|-1\}$, and satisfies $v>u_{i+1}$ in the lexicographic order on $\mathcal U$. Therefore{
\begin{align*}
|\LR(u,t)|&=\sum_{j=0}^{|u|-1} \l(c_t(u_j)-1\r)+1 =\sum_{i\ge 0} M_i(u,t)(i-1)+1\\
|\R(u,t)|&= \sum_{j=0}^{|u|-1} (c_t(u_j)-i_{j+1}).
\end{align*}}
Let $\tilde u_1=\varnothing<\tilde u_2<\dots< \tilde u_{|t|}$ be the nodes of  $t$, in increasing lexicographic order. Then {
\begin{equation}\label{eq:H_and_S}
 H_t(k)=|\tilde u_{k+1}| \textrm{~~and~~}  S_t(k)=|\R(\tilde u_k,t)|+c_t(\tilde u_k)-1,
\end{equation}} so that the discrepancy between $H_t$ and $S_t$ can be accessed using the number of nodes to the right of the paths to $\tilde u_i$, $i= 1, \dots, |t|$. This observation lies at the heart of our approach.

The set of plane trees with degree sequence $\ds$ and a distinguished node (marked plane trees) is denoted by $\Trees^{\bullet}_\ds=\l\{(t,u)~: t\in \Trees_\ds, u\in t\r\}$, and 
the uniform distribution on this set is denoted $`P^{\bullet}_\ds$. Under $`P^\bullet_\ds$, a marked tree $(t,u)$ is distributed as $(t',u')$ where $t'$ is a tree sampled under $`P_\ds$ and $u'$ is a uniformly random node in $t'$.
We now decompose a marked tree $(t,u)$ along the branch $\cro{\varnothing,u}$. 
First, consider the structure of this branch, that we call the contents:
\[\Cont(t,u):=\big((c_t(u_0),i_1),\dots,(c_t(u_{|u|-1}),i_{|u|})\big).\]
We write $J^\bm$ for the set of potential vectors $\Cont(t,u)$ when the composition of the branch $\cro{\varnothing, u}$ is $\bM(u,t)=\bm$.
Besides, notice that 
\bq\label{eq:cjm}
|J^{\bm}|=\binom{|\bm|}{(m_i,i\geq 1)}\prod_{i\geq 1} i^{m_i}. 
\eq
Since, if $\Cont(u,t)\in J^{\bm}$ { then $|\LR(u,t)|=1+\sum_{i\ge 0}(i-1)m_i$}, we will use the following notation:
\[|\LR(\bm)|:=1+\sum_{i\ge 0}(i-1)m_i.\]

\medskip
\noindent\textsc{The forest off a distinguished path.}\
For a tree $t$ and any node $v\in t$, let $t_v=\{w~:vw \in t\}$ be the subtree of $t$ rooted at $v$. The sequence of trees $\F(t,u) = (t_v, v \in \LR(u,t))$
is the forest constituted by the subtrees of $t$ rooted at the vertices belonging to $\LR(u,t)$, and sorted according to the rank of their root for the lexicographic order.

The decomposition which associates $(\Cont(t,u),\F(t,u))$ to a marked tree $(t,u)$ is clearly one-to-one. The following proposition characterises the distributions of $\bM(\bu,\bt)$, $\Cont(\bu,\bt)$, and $|\R(\bu,\bt)|$ when $(\bt,\bu)$ is sampled under $`P^\bullet_\ds$. { In the following, for two sequences of integers $\bs=(n_0,n_1,\dots)$ and $\bm=(m_0,m_1,\dots)$ we write $\bs-\bm=(n_0-m_0,n_1-m_1,\dots)$.}

\begin{pro}\label{pro:fond-comb} Let $\ds=(n_0,n_1,\dots)$ be a degree sequence and let ${\bm}=(m_0,m_1,\dots)$ be such that $m_0=0$, and $m_i\leq n_i$ for any $i\geq 1$. Let $(\bt,\bu)$ be chosen according to $`P^{\bullet}_\ds$.\\
(a) We have  
\[`P^{\bullet}_\ds\l(\bM(\bu,\bt)=\bm\r)=\frac{|\LR(\bm)|\,|\bm|! \,|\ds-\bm|!}{|\ds|!\,|\ds-\bm|}\cdot \prod_{i\geq 1}\binom{n_i}{m_i}{i^{m_i}}.\]
(b) Moreover, for any vector $C\in J^\bm$, 
\[`P^{\bullet}_\ds\l(\Cont(\bu,\bt)=C~|~\bM(\bu,\bt)=\bm\r)= 1/{\#J^\bm}.\]
(c) For any $x\geq 0$, and $\bm$ such that $`P_\ds^\bullet(\bM(\bu,\bt)=\bm)>0$, 
\bq\label{eq:rep}
`P^{\bullet}_\ds\l(\l||\R(\bu,\bt)|-\frac{\sigma_\ds^2}{2}|\bu|\r|\geq x~\Bigg|~\bM(\bu,\bt)=\bm\r)=`P\l(\l|\sum_{j\geq 1}\sum_{k=1}^{m_j} U^{(k)}_j-\frac{\sigma_\ds^2}{2}|\bm|\r|\geq x\r)
\eq
where the $U^{(k)}_j$ are independent random variables, $U^{(k)}_j$ is uniform in $\{0,\dots,j-1\}$ and where $\sigma_\ds^2$ is the variance associated with $(p_i=n_i/|\ds|$, $i\ge 0$) (as done on \eref{eq:sigma}).
\end{pro}
\begin{proof}Since the backbone decomposition is a bijection, we have for any 
vector $C\in J^\bm$, we have
\be
`P^{\bullet}_\ds\l(\Cont(\bu,\bt)=C\r)&=&\frac{\#\Forests_{\ds-\bm}}{|\ds|\cdot \#\Forests\ds}\\
&=& \frac{|\LR(\bm)|}{|\ds-\bm|}\binom{|\ds-\bm|}{(n_i-m_i,i\geq 0)} \bigg/\binom{|\ds|}{(n_i,i\geq 0)},
\ee
by the expression for the number of forests in \eref{eq:nbf}. As $`P^\bullet_\ds\l(\Cont(\bu,\bt)=C\r)$ is independent of $C\in J^\bm$, it suffices to multiply by $\#J^{\bm}$ in order to get $`P^{\bullet}_\ds\l(\bM(\bu,\bt)=\bm\r)$. After simplification, this yields the first statement in (a), and then (b). Now, (b) implies that 
for any $R\geq 0$, and any composition $\bm$ for which $`P^{\bullet}_\ds(\bM(\bu,\bt)=\bm)>0$, we have
\[`P^{\bullet}_\ds\l(|\R(\bu,\bt)|=R~|~\bM(\bu,\bt)=\bm\r)=`P\l(\sum_{j\geq 1}\sum_{k=1}^{m_j} U^{(k)}_j=R\r),\]
where the $U^{(k)}_j$ are independent random variables, and $U^{(k)}_j$ is uniform in $\{0,\dots,j-1\}$. This implies assertion (c) and completes the proof.
\end{proof}

\section{Convergence of uniform trees to the CRT: Proof of Theorem~\ref{thm:main_encodings}}
\label{sec:proof}

\subsection{The general approach}\label{sec:conv_approach}
Our approach uses the phenomenon observed in Marckert \& Mokkadem \cite{MA-MO} in the case of critical Galton--Watson tree (having a variance): under some mild assumptions the \LL path $S_\bt$ and the height process $H_\bt$ are asymptotically proportional, that is, up to a scalar normalisation, the difference between these processes converge to the zero function. It turns out that a similar phenomenon occurs when the degree sequence is prescribed, and this is the basis of our proof.

In order to prove Theorem~\ref{thm:main_encodings} we proceed in two steps: the first one consists in showing that the depth-first walk $S_\bt$ associated to a tree sampled under $`P_{\ds(\kappa)}$ converges to a Brownian excursion. The process $S_\bt$ is much easier to deal with than $H_\bt$, { since $S_\bt$ is \emph{essentially} a random walk conditioned to stay non-negative, and forced to end up at the origin (precisely at $-1$).} We provide the details in Section~\ref{sec:luka} below. The core of the work lies in the second step, which consists in proving that { rescaled versions of } $S_\bt$ and $H_\bt$ are indeed close, uniformly on $[0,1]$. More precisely, by Theorem 3.1 p.~27 of \cite{BIL}, the following proposition is sufficient to show that $\nk^{-1/2} 2 S_\bt(\nk\cdot)$ and $\nk^{-1/2} \sigma_{\kappa}^2 H_\bt((\nk-1)\cdot )$ have the same limit in $({\mathcal C}[0,1], \|_\infty)$.

\begin{pro}\label{pro:compar_unif}Under the hypothesis of Theorem \ref{thm:main_gh}, there exists $c_\kappa=o(\nk^{1/2})$ such that, as $\kappa\to\infty$,
\[`P_{\ds(\kappa)}\l( \sup_{x\in[0,1]}\l|S_\bt(x\nk)- \frac{\sigma_{\kappa}^2}2 H_\bt(x(\nk-1))\r|\geq c_\kappa \r)\xrightarrow[\kappa\to\infty]{} 0.\]
\end{pro}
In order to prove Proposition~\ref{pro:compar_unif}, { recall the representations of $S_t$ and $H_t$ in terms of $|\R(u,t)|$ and $|u|$ given in \eref{eq:H_and_S}. }
A non-uniform version of the claim in Proposition~\ref{pro:compar_unif} is the following: 
\begin{pro}\label{pro:compar}Assume the hypothesis of Theorem \ref{thm:main_gh}. Let $(\bt,\bu)$ chosen under $`P_{\ds(\kappa)}^\bullet$. There exists $c_\kappa=o(\nk^{1/2})$ such that, 
\[`P_{\ds(\kappa)}^\bullet\l( \l||\R(\bu,\bt)|- \frac{\sigma_{\kappa}^2}2|\bu|\r|\geq c_\kappa \r)\xrightarrow[\kappa\to\infty]{} 0.\]
\end{pro}
{ Again, by \eref{eq:H_and_S}, one sees that 
{ $$\l|\left||\R(\bu,\bt)|- \frac{\sigma_{\kappa}^2}2|\bu|\right|-\l|S_\bt(\bu)- \frac{\sigma_{\kappa}^2}2 H_\bt(\bu-1)\r|\r|\le \Delta_\kappa,$$
and $\Delta_\kappa=o(\sqrt \nk)$, by assumption.} Therefore, Proposition~\ref{pro:compar} implies then that the proportion of indexes $m\in\cro{0,\nk}$ for which $$\l|S_\bt(m+1)- \frac{\sigma_{\kappa}^2}2 H_\bt(m)\r|\geq c_\kappa{ -}\Delta_k$$ goes to 0 { (we will choose $c_\kappa$ such that $\Delta_\kappa=o(c_\kappa)$)}. In this case, if the sequence of processes $(D_\kappa:=\nk^{-1/2}( S_\bt(x\nk)- \frac{\sigma_{\kappa}^2}2 H_\bt(x(\nk-1))),\kappa\geq 1)$ is tight, we can deduce the convergence of the finite distributions of $(D_\kappa,\kappa\geq 1)$  to those of the null process on $[0,1]$.  Hence,  to show Proposition~\ref{pro:compar_unif}, it suffices to show Proposition~\ref{pro:compar} together with the tightness of $(D_\kappa,\kappa\geq 1)${ ; the tightness is actually also} needed to show the convergence of $D_\kappa$ in distribution in $C[0,1]$ (see \cite[see, e.g.][]{BIL}). Since under  the sequence of distributions $`P_{\ds(\kappa)}$, 
the family of rescaled versions of $S_\bt$ (see Section~\ref{sec:luka}) is tight, it suffices to prove that the family of rescaled versions of $H_\bt$ is tight as well.\par
We need also to say a word about the fact that both processes $S_t$ and $H_t$ have a small difference in their time rescaling. Again, this is not a problem since the process $S_t$ has its increments bounded by $\Delta_\kappa=o(\sqrt(\nk))$.}

\medskip
\noindent\textbf{Remark.}\ Under slightly stronger assumptions on the degree sequences, it is possible to control the discrepancy between the height process and the \LL path at \emph{every point} in $\{0,1,\dots, \nk-1\}$. More precisely it would be possible to show that
\begin{equation}\label{eq:control_everywhere}
`P_{\ds(\kappa)}^{\bullet}\l( \l||\R(\bu,\bt)|- \frac{\sigma_{\kappa}^2}2|\bu|\r|\geq c_{\kappa} \r) =o(1/\nk).
\end{equation}
Using the union bound, this yields the convergence of the rescaled height process to a Brownian excursion, as a random function in ${\mathcal C}[0,1]$. One is easily convinced that with the optimal assumptions for Theorem~\ref{thm:main_encodings}, the bound in \eqref{eq:control_everywhere} might just not hold.

\medskip
We now move on to the ingredients of the proof: we first give the details of the convergence of $\nk^{-1/2}S_\bt(~\cdot~\nk)$ to a Brownian excursion in Section~\ref{sec:luka}, then we prove tightness for $\nk^{-1/2}H_\bt~(\cdot~(\nk-1))$ in Section~\ref{sec:tightness}. The longer proof of Proposition~\ref{pro:compar} is delayed until Section~\ref{sec:fdd}.

\subsection{Convergence of the \LL walk}\label{sec:luka}

In this section, we give the details of the proof of the convergence of the depth-first walk under $`P_{\ds(\kappa)}$ towards the Brownian excursion.
\begin{lem}\label{lem:exchan}Assume the hypothesis of Theorem \ref{thm:main_gh}. Under $`P_{\ds(\kappa)}$,
$$\left(\frac{S_\bt(x \nk)}{\sigma_\kappa \nk^{1/2}}\right)_{x\in[0,1]} \xrightarrow[\kappa\to+\infty]{(law)} \se$$
as random functions in $\mathcal C[0,1]$.
\end{lem}
\begin{proof}{ 
Let $\bc=\{c_1,c_2,\dots, c_{\nk}\}$ be a multiset of $\nk$ integers whose distribution is given by $\ds(\kappa)$. Let $\pi=(\pi_1,\pi_2,\dots, \pi_{\nk})$ be a uniform random permutation of $\{1,2,\dots, \nk\}$, and  for $j\in\{1,\dots, \nk\}$, define 
$$W_{\pi}(j)=\sum_{i=1}^j (c_{\pi_i}-1).$$
Theorem 20.7 of Aldous \cite{Aldous1983} (see also Theorem~24.1 in \cite{BIL}) ensures that, when $\Delta_\kappa =o(\sqrt{\nk})$,
$$\left(\frac{W_\pi(s \nk)}{\sigma_\kappa \nk^{1/2}}\right)_{s\in [0,1]}\xrightarrow[\kappa\to+\infty]{(law)} \sb,$$
in $\mathcal C[0,1]$, where $\sb=(\sb(s), s\in [0,1])$ is a standard Brownian bridge.   \par
The increments of the walk $(W_\pi(j), 0\le j\le \nk)$ satisfy $c_{\pi_i}-1\ge -1$ for every $i$ (such walks are sometimes called \emph{left-continuous}), and furthermore, $W_\pi(\nk)=-1$. The cycle lemma \cite{DvMo1947a} ensures that there is a unique way to turn the process $W_\pi$ into an excursion by shifting the increments cyclically (in each rotation class there is a unique excursion) : to see this, first extend the definition of the permutation, setting $\pi_j:=\pi_{j-\nk}$  for any  $j\in\{\nk+1,\dots, 2\nk\}$. For $j_\pi$ the location of the first minimum of the walk $W_\pi$ in $\{1,\dots, \nk\}$, we have that $W_\pi(j+j_\pi)-W_{\pi}(j_\pi)$ is an excursion in the following sense:
$$ \tilde S_\pi(j):=W_\pi(j+j_\pi)-W_\pi(j_\pi)\ge 0 \quad \text{for } j<\nk\qquad \text{and }\tilde S_\pi(\nk)=-1.$$
Since in each rotation class there is exactly one excursion, and since the set of excursions hence obtained is exactly the set of  depth-first walk of the trees in $\Trees_{\ds(\kappa)}$, it is then easy to conclude that for $\bt$ uniformly chosen in $\Trees_{\ds(\kappa)}$, 
$$(S_\bt(j), 0\le j\le \nk) \eqdist (\tilde S_\pi(j), 0\le j\le \nk),$$
for $\pi$ a random permutation of $\{1,\dots, \nk\}$. } 
Since the Brownian bridge $\sb$ has almost surely a unique minimum, the claim follows by the mapping theorem \cite{BIL}.
\end{proof}

\subsection{Tightness for the height process}\label{sec:tightness}

The rescaled height process under $`P_\ds(\kappa)$ is the process in $\mathcal C[0,1]$, $h_\kappa=\nk^{-1/2} H(~\cdot~(\nk-1))$. 
In this section, we prove that the family $(h_\kappa,\kappa >0)$ is tight  (we will omit the $\kappa$ when unnecessary). Since $h_\kappa(0)=0$, the following lemma is sufficient to prove tightness \cite[see, e.g.,][]{BIL}. \par
Let $\omega_h$ be the modulus of continuity of the rescaled height process $h$: for $\delta>0$
$$\omega_h(\delta)=\sup_{|t-s|\le \delta} |h(s)-h(t)|.$$
\begin{lem}\label{lem:tightness}Under the hypothesis of Theorem \ref{thm:main_gh}, for any $\epsilon>0$ and $\eta>0$, there exists $\delta>0$ such that, for all $\kappa$ large enough,
$$`P_{\ds(\kappa)}(\omega_h(\delta)>\epsilon)<\eta.$$
\end{lem}

The bound we provide consists in reducing the bounds on the variations of $h$ to bounds on the variations of the \LL path $S$, which is known to be tight since it converges in distribution (Lemma~\ref{lem:exchan}). The underlying ideas are due to \citet{AdDeJa2010a} and \citet{Addario2011a} to prove Gaussian tail bounds for the height and width of Galton--Watson trees and random trees with a prescribed degree sequence, respectively. 

For a plane tree $t\in \Trees$, let $t^-$ be the mirror image of $t$, or in other words, the tree obtained by flipping the order of the children of every node. Then, we let $S^-_t:=S_{t^-}$ be the \emph{reverse depth-first walk}. Observe that the mirror flip is a bijection, so that $S_t$ and $S^-_t$ have the same distribution under $`P_{\ds(\kappa)}$.

\begin{proof}[Proof of Lemma~\ref{lem:tightness}]In this proof, we identify the nodes of a tree $t$ and their index in the lexicographic order; so in particular, we write $H_t(u)$ for the height of a node $u$ in $t$, and we write $|u-v|\le \delta \nk$ to mean that $u$ and $v$ are within $\delta \nk$ in the lexicographic order (that is, $u=\tilde u_i$ and $v=\tilde u_j$ for some $i$ and $j$ satisfying $|i-j|\le \delta \nk$). 
	
  Consider a tree $t$ and two nodes $u$ and $v$. Write $u\wedge v$ for the (deepest) first common ancestor of $u$ and $v$ in $t$. In the following we write $u\preceq v$ to mean that $u$ is an ancestor of $v$ in $t$ ($u=v$ is allowed). Then, 
\begin{align}
|H_t(u)-H_t(v)| 
&\le |H_t(u)-H_t(u\wedge v)| + |H_t(v)-H_t(u\wedge v)|,
\end{align}
so that it suffices to bound variations of $H_t$ between two nodes on the same path to the root:
\begin{align*}
\sup_{|u-v|\le \delta \nk} |H_t(u)-H_t(v)| \le  { 2 + }2\sup_{w\preceq u,|u-w|\le \delta  \nk} |H_t(u)-H_t(w)|.
\end{align*}
{ (The extra two in the previous bound is needed because of the following reason: the closest common ancestor $u\wedge v$ might not be within distance $\delta \nk$ of either $u$ and $v$; however, there is certainly a node $w$ lying within distance one of $u\wedge v$ that is visited between $u$ and $v$.)}
Now, observe that, for $w\preceq u$, every node $v$ on the path between $w$ and $u$ which has degree more than one contributes at least one to the number of nodes off the path between $w$ and $u$:
$$1+\sum_{w\preceq v\preceq u} (c_t(v)-1)\ge H_t(u)-H_t(w)-\sum_{w\preceq v\preceq u} \I_{\{c_t(v)=1\}}$$
However, one may also bound this same number of nodes in terms of the depth-first walk $S_t$, and the reverse depth-first walk $S^-_t$:
\bq\label{eq:bo1}
1+\sum_{w\preceq v\preceq u} (c_t(v)-1)\le S_t(u)-S_t(w)+S_t^-(u)-S_t^-(w)+{ 2}c_t(w).
\eq
In other words, we have
\begin{align}\label{eq:bound_modulus}
\sup_{|u-v|\le \delta  \nk} |H_t(v)-H_t(u)|
&\le { 2 + } 2\sup_{|u-w|\le \delta  \nk, w\preceq u} |S_t(u)-S_t(w)| + 2\sup_{|u-w|\le \delta  \nk, w\preceq u} |S_t^-(u)-S_t^-(w)|\nonumber\\
&\quad+{ 2}\max_{w} c_t(w) + \sup_{|u-w|\le \delta  \nk} \sum_{w\preceq v\preceq u} \I_{\{c_t(v)=1\}}\nonumber\\
&\le { 2 + } 2\sup_{|u-w|\le \delta n} |S_t(u)-S_t(w)| + 2\sup_{|u-w|\le \delta  \nk} |S_t^-(u)-S_t^-(w)|\nonumber\\
&\quad + { 2}\Delta_\kappa + \sup_{|u-w|\le \delta  \nk} \sum_{w\preceq v\preceq u} \I_{\{c_t(v)=1\}}\nonumber\\
&\le { 2 + }2  \nk^{1/2} \omega_s(\delta) +2  \nk^{1/2} \omega_{s-}(\delta)+ { 2}\Delta_\kappa + \sup_{|u-w|\le \delta  \nk} \sum_{w\preceq v\preceq u} \I_{\{c_t(v)=1\}},
\end{align}
where $\omega_s$ and $\omega_{s^-}$ denote the moduli of continuity of the rescaled \LL path $ \nk^{-1/2} S_t$ and $ \nk^{-1/2} S^-_t$, respectively.

The first four terms in \eqref{eq:bound_modulus} are easy to bound since $\Delta_\kappa=o(\sqrt \nk)$ and, { after renormalisation}, $S_t$ and $S^-_t$ are tight under $`P_{\ds(\kappa)}$. The only term remaining to control is the one concerning the number of nodes of degree one:
$$Y_t(\delta):=\sup_{|u-w|\le \delta  \nk} \sum_{w\preceq v\preceq u} \I_{\{c_t(v)=1\}}.$$
To bound $Y_t(\delta)$ we relate the distribution of trees under $`P_{\ds(\kappa)}$ to those under $`P_{\ds(\kappa)^\star}$, where $\ds(\kappa)^\star=(n_0^\star, n_1^\star, \dots)$ is obtained from $\ds(\kappa)$ by removing all nodes of degree one, i.e., $n_1^\star=0$ and $n_i^\star=n_i$ for every $i\neq 1$. Then, in a tree $t^\star$ sampled under $`P_{{\ds(\kappa)}^\star}$, one has $Y_{t^\star}(\delta)=0$. Recall also that $\Delta_\kappa=o(\sqrt \nk)$.
Now, for a sum of three terms to be at least $\epsilon$, at least one term must exceed $\epsilon/3$. So for every $\epsilon,\delta>0$, there exists a $\delta>0$ such that, for all $\kappa$ large enough,
\begin{align*}
`P_{\ds(\kappa)^\star}(\omega_h(\delta)\ge \epsilon) 
&\le `P_{\ds(\kappa)^\star}(2\omega_s(\delta)>\epsilon/3)+`P_{\ds(\kappa)^\star}(2\omega_{s^-}(\delta)>\epsilon/3)\\
&= 2 `P_{\ds(\kappa)^\star}( 6\omega_s(\delta)\ge \epsilon )<\eta,
\end{align*}
since, under $`P_{\ds(\kappa)^\star}$, $S_t$ and $S^-_t$ have the same distribution and $\nk^{-1/2}S_t$ is tight, and since $`P_{\ds(\kappa)^\star}(\Delta_\kappa\nk^{-1/2}\geq `e/3)$ is zero for $\kappa$ large enough.
This proves that $\nk^{-1/2}H_t$ is tight under $`P_{\ds(\kappa)^\star}$.

Now, we can couple the trees sampled under $`P_{\ds(\kappa)^\star}$ and $`P_\ds(\kappa)$. Since the nodes of degree one do not modify the tree structure, a tree $t$ under $`P_{\ds(\kappa)}$ may be obtained by first sampling $t^\star$ using $`P_{\ds(\kappa)^\star}$, and then placing the nodes of degree one uniformly at random : precisely, this insertion of nodes is done inside the edges of $t^\star$ (plus a phantom edge below the root).  Given any ordering of the edges of $t^\star$ (plus the one below the root), the vector $(X_1^\star,\dots, X_{\nk-n_1(\kappa)}^\star)$ of numbers of nodes of degree one falling in these edges 
is such that
$$(X_1^\star,\dots, X_{\nk-n_1(\kappa)}^\star) \eqdist \text{Multinomial}\l(n_1(\kappa); \frac{1}{\nk-n_1(\kappa)},\dots,\frac{1}{\nk-n_1(\kappa)}\r).$$

Conversely, $t^\star$ is obtained from $t$ by removing the nodes of degree one, so that $t$ and $t^\star$ can be thought as random variables in the same probability space under $`P_{\ds(\kappa)}$. To bound $Y_t(\delta)$, observe that it is unlikely that adding the nodes of degree one in this way creates too long paths.\par
In fact, ``the length of paths'' is expected to be multiplied by $1+q_\kappa$ for
$q_\kappa=n_1(\kappa)/(\nk-n_1(\kappa))$. Let $\alpha=2+q_\kappa$, and fix $\delta>0$ such that $`P_{\ds(\kappa)^\star}(\omega_h(\delta)\ge \epsilon/\alpha)<\eta/2$; such a $\delta>0$ exists since the height process is tight under $`P_{\ds(\kappa)^\star}$. Note that since we add nodes in the construction of $t$ under $`P_{\ds(\kappa)}$ from $t^\star$ under $`P_{\ds(\kappa)^\star}$, nodes that are within $\delta \nk$ in $t$ are also within $\delta \nk$ in $t^\star$. Write $h^\star$ for the rescaled height process obtained from $t^\star$, the tree associated with $t$ by deletion of all nodes of degree one (the rescaling stays $\sqrt \nk$). We have,
\begin{align*}
`P_{\ds(\kappa)}(\omega_h(\delta)\ge \epsilon)
&\le `P_{\ds(\kappa)}(\omega_{h^\star}(\delta)\ge \epsilon/\alpha) + `P_{\ds(\kappa)}(\omega_h(\delta)\ge \epsilon~,~\omega_{h^\star}(\delta)\le\epsilon/\alpha)\\
&\leq `P_{\ds(\kappa)^\star}(\omega_h(\delta)\ge \epsilon/\alpha) + `P_{\ds(\kappa)}(\omega_h(\delta)\ge \epsilon~|
~\omega_{h^\star}(\delta)\le\epsilon/\alpha)\\
&\le `P_{\ds(\kappa)^\star}(\omega_h(\delta)\ge \epsilon/\alpha)+\delta \nk^2 `P\l(\sum_{i=1}^{\epsilon \sqrt{\nk}/\alpha} (1+X_i^\star) \ge `e \sqrt{\nk}\r)\\
&\le \eta/2 + \delta \nk^2 `P\l(\sum_{i=1}^{\epsilon\sqrt{\nk}/\alpha} X_i \ge `e\sqrt{\nk}(1-1/\alpha)\r),
\end{align*} 
where the $X_i$ are i.i.d.\ Binomial$(n_1, 1/(\nk-n_1))$ random variables. The last line follows from the standard fact that the numbers $(X_i^\star)$ obtained from a sampling\emph{ without replacement} (of the $n_1(\kappa)$ nodes of degree one) are more concentrated than their counterpart $(X_i)$ coming from a sampling \emph{with replacement} \cite{Aldous1983}. 

Now, the sum in the right-hand side is itself a binomial random variable:
$$\sum_{i=1}^{\epsilon\sqrt{\nk}/\alpha} X_i \eqdist\text{Binomial}\l(\epsilon\sqrt{\nk} n_1/\alpha , \frac 1{\nk-n_1}\r)$$
whose mean is $`e\sqrt{n_k}q_\kappa/(2+q_\kappa)$ when $`e\sqrt{\nk}(1-1/\alpha)=`e\sqrt{\nk}(1+q_\kappa)/(2+q_\kappa)$. By Chernoff's bound, using that  $q_\kappa$ converges, it follows that for some constant $c>0$ valid for $\kappa$ large enough, 
$$`P_{\ds(\kappa)}(\omega_h(\delta)\ge \epsilon)\le \eta/2 + \delta \nk^2 e^{-c \sqrt{\nk}/\epsilon}.$$
Finally, for all $\kappa$ large enough, with this value for $\delta$, we have $`P_{\ds(\kappa)}(\omega_h(\delta)\ge \epsilon)<\eta$, which completes the proof.
\end{proof}

\section{Finite dimensional distributions: Proof of Proposition~\ref{pro:compar}}\label{sec:fdd}

\subsection{A roadmap to Proposition~\ref{pro:compar}: identifying the bad events}
Our approach consists in showing that if the event in Proposition~\ref{pro:compar} occurs, then one of the following three events must occur: (1) either the depth $|\bu|$ of node $\bu$ is unusually large, (2) or the content of the branch $\llbracket \varnothing, \bu\rrbracket$ is atypical, (3) or the number of nodes to the right of the path is not what it should be, despite of the length $|\bu|$ and content $\bM(\bu,\bt)$ being typical.

We will then prove that those simpler events are unlikely. For $h\geq 0$, and two sequences $a=(a_\kappa,\kappa\ge 0)$, and $b=(b_\kappa,\kappa\ge 0)$ we define families of sets $A_{h,a,b}$ as follows. Given a sequence of degree distribution $(\ds(\kappa),\kappa \geq 0)$,
\begin{align*}
A_{h,a,b}(\kappa):=\l\{\bm~ : 
      |\bm|=h, \l|\l(\sum_{i\ge 0} m_i\frac{i-1}2\r)-\frac{h\sigma^2_{\kappa}}{2}\r|\leq a_{\kappa}, 
   \sum_{i\ge 1} m_i i^2 \leq b_{\kappa}\r\}.\nonumber
\end{align*}

If $\bm\in A_{h,a,b}(\kappa)$ then $|\bm|=h$, and $\bm$ corresponds to the content of a branch $\cro{\varnothing,u}$ such that $|u|=h$. The set $A_{h,a,b}(\kappa)$ are designed to contain most typical contents of a branch of length $h$ under $`P_{\ds(\kappa)}$, provided the choices for the sequences $a$ and $b$ are suitable. The decomposition of the bad event we have outlined above is then expressed formally by
\begin{align}\label{eq:a-control}
`P_{\ds(\kappa)}^{\bullet}\l(\l||\R(\bu,\bt)|-\frac{\sigma_\kappa^2}{2}|\bu|\r|\geq c_\kappa\r)&\leq
 `P_{\ds(\kappa)}^{\bullet}(|\bu|\geq x \sqrt{\nk})\nonumber\\
&\quad+~`P_{\ds(\kappa)}^\bullet\l(|\LR(\bu,\bt)|\ge x \sqrt{\nk}\r)\nonumber\\
&\quad+~`P_{\ds(\kappa)}^{\bullet}\l(|\bu|\vee |\LR(\bu,\bt)|\le x \sqrt{\nk}, \bM(\bu,\bt)\notin \bigcup_{h\leq x\sqrt{\nk}}A_{h,a,b}(\kappa)\r)\nonumber\\
&\quad+\sum_{h\leq x\sqrt{\nk} \atop{\bm \in A_{h,a,b}(\kappa)}}`P_{\ds(\kappa)}^{\bullet}\l(\l||\R(\bu,\bt)|-\frac{\sigma_\kappa^2}{2}|\bu|\r|\geq c_\kappa, \bM(\bu,\bt)=\bm\r).
\end{align}
{ Proving Proposition~\ref{pro:compar} reduces to proving that every term in the right-hand side above can be made arbitrarily small for large $\kappa$ by a judicious choice of $a_\kappa$, $b_\kappa$, $c_\kappa$ and $x$.} The bound on the first term is a direct consequence of the Gaussian tail bounds for the height of trees recently proved by \citet{Addario2011a} in the very setting we use:
\begin{equation}\label{eq:bound_height}
`P_{\ds(\kappa)}^{\bullet}(|\bu|\geq x \sqrt{\nk})\le `P_{\ds(\kappa)}\l(\max_{u\in t} |u|\geq x\sqrt{\nk}\r)\le \exp(-c x^2 /\sigma_\kappa^2),
\end{equation}
for a universal constant $c>0$ and all sufficiently large $\kappa$. The second term is bounded using the depth-first walk $S$ and the reverse depth-first walk $S^-$, as in the proof of Lemma~\ref{lem:tightness}:
\begin{align*}
`P_{\ds(\kappa)}^\bullet(|\LR(\bu,\bt)|\ge x\sqrt{\nk})
&\le `P_{\ds(\kappa)}\l(\max_{0\le k\le \nk}\{S(k)+S^-(k)\}+\Delta_\kappa \ge x\sqrt{\nk}\r)\\
&\le 2 `P_{\ds(\kappa)}\l(\max_{0\le k \le \nk} S(k)\ge \frac x 3\sqrt{\nk}\r),
\end{align*}
for all $\kappa$ large enough, since $\Delta_\kappa=o(\nk)$ and $S$ and $S^-$ have the same distribution under $`P_{\ds(\kappa)}$. We finish using the tightness of $\nk^{-1/2}S(\nk.)$ under $`P_{\ds(\kappa)}$; more precisely, we have
\begin{align}\label{eq:bound_LR}
`P_{\ds(\kappa)}^\bullet(|\LR(\bu,\bt)|\ge x\sqrt{\nk}) \le 16 \cdot 9 \cdot \frac{\sigma_\kappa^2}{x^2},
\end{align}
by Lemma~20.5 of \cite{Aldous1983}. The bounds on the two remaining terms are stated in Lemmas~\ref{lem:content} and~\ref{lem:discrepancy}, the proof of which appear in Sections~\ref{sec:content} and~\ref{sec:discrepancy}, respectively. 

\begin{lem}\label{lem:content}Since { $\Delta_k=o(\sqrt \nk)$} there exists $`e_\kappa$ such that $\Delta_\kappa\le\varepsilon_\kappa \sqrt{\nk}$, with $0<\varepsilon_\kappa\to 0$. Let $a_\kappa=\varepsilon_\kappa^{1/4} \sqrt{\nk}$ and $b_\kappa=\varepsilon_\kappa^{1/2} \nk $. Then, for every $x>0$, and all $\kappa$ large enough,
\begin{align*}
`P_{\ds(\kappa)}^{\bullet}\l(|\bu|\vee |\LR(\bu,\bt)|\leq x \sqrt{\nk}, \bM(\bu,\bt)\not\in \bigcup_{h\leq x \sqrt{\nk}} A_{h,a,b}(\kappa)\r)
&\le 6x^2 e^{x^2}
\exp\l(-\frac {\varepsilon_\kappa^{-1/2}}{2x(\sigma^2_\kappa+1)+2}\r).
\end{align*}
\end{lem}

\begin{lem}\label{lem:discrepancy}Since { $\Delta_k=o(\sqrt \nk)$} there exists $`e_\kappa$ such that $\Delta_\kappa\le\varepsilon_\kappa \sqrt{\nk}$, with $0<\varepsilon_\kappa\to 0$ and $\varepsilon_\kappa^{-3/4}=o(\nk)$ as $\kappa\to\infty$. Let $a_\kappa=\varepsilon_\kappa^{1/4} \sqrt{\nk}$, $b_\kappa=\varepsilon_\kappa^{1/2} \nk$, and $c_\kappa=\varepsilon_\kappa^{1/8}\sqrt{\nk}$. 	Then, for all $\kappa$ large enough,
\begin{equation}\label{eq:devi}
\sum_{h\leq x\sqrt{\nk}\atop{\bm \in A_{h,a,b}(\kappa)}}
`P_{\ds(\kappa)}^{\bullet}\l(\l||\R(\bu,\bt)|-\frac{\sigma_\kappa^2}{2}|\bu|\r|\geq c_\kappa, \bM(\bu,\bt)=\bm\r)\le 2 e^{-\varepsilon_\kappa^{-1/2}}.
\end{equation}
\end{lem}

Before proceeding with the proofs of these two lemmas, we indicate how to use them in order to complete the proof of Proposition~\ref{pro:compar}. Let $\varepsilon_\kappa$ be such that $\Delta_\kappa\le \varepsilon_\kappa \sqrt{\nk}$, with $\varepsilon_\kappa\to0$ as $\kappa\to\infty$. Then, set $a_\kappa=\varepsilon_\kappa^{1/4} \sqrt{\nk}$, $b_\kappa=\varepsilon_\kappa^{1/2} \nk$ and $c_\kappa=\varepsilon_\kappa^{1/8} \sqrt{\nk}$. Let now $\epsilon>0$ be arbitrary. Pick $x>0$ large enough such that, for all $\kappa$ large enough,
\begin{align*}
`P_{\ds(\kappa)}^{\bullet}(|\bu|\geq x \sqrt{\nk}) + `P_{\ds(\kappa)}^\bullet(|\LR(\bu,\bt)|\ge x\sqrt{\nk}) < \epsilon/2.
\end{align*}
The bounds in \eqref{eq:bound_height} and \eqref{eq:bound_LR}, and the fact that $\sigma^2_\kappa \to \sigma^2_\bp$ ensure that this is possible. The value for $x$ being fixed, Lemmas~\ref{lem:content} and~\ref{lem:discrepancy} now make it possible to choose $\kappa_0$ large enough such that, for all $\kappa\ge \kappa_0$, the two remaining terms in the right-hand side of \eqref{eq:a-control} also sum to at most $\epsilon/2$. Thus, for all $\kappa\ge \kappa_0$, we have
$$`P_{\ds(\kappa)}^{\bullet}\l(\l||\R(\bu,\bt)|-\frac{\sigma_\kappa^2}{2}|\bu|\r|\geq c_\kappa\r)<\epsilon,$$
which completes the proof, since $\epsilon$ was arbitrary.

\subsection{The content of a branch is very likely typical: Proof of Lemma~\ref{lem:content}}\label{sec:content}

We now prove that, on the event that $|\bu|$ and $|\LR(\bu,\bt)|$ are not too large, the content of the branch $\cro{\varnothing, \bu}$ is typical with high probability. 

We start by rewriting the probability of interest using Proposition \ref{pro:fond-comb}:
\begin{align}\label{eq:content1}
&`P_{\ds(\kappa)}^{\bullet}\l(|\bu|\vee |\LR(\bu,\bt)|\leq x\sqrt{\nk}, \bM(\bu,\bt)\not\in \bigcup_{h\leq x\sqrt{\nk}} A_{h,a,b}(\kappa)\r)\nonumber\\
&=\sum_{h\leq x\sqrt{\nk}} `P_{\ds(\kappa)}^{\bullet}\l(|\bu|=h, |\LR(\bu,\bt)|\le x\sqrt{\nk},\bM(\bu,\bt)\not\in  A_{h,a,b}(\kappa)\r)\nonumber\\
&=\sum_{h\le x\sqrt{\nk}} \sum_{|\bm|=h\atop{\bm\not\in A_{h,a,b}(\kappa)}, |\LR(\bm)|\le x\sqrt{\nk}}`P_{\ds(\kappa)}^{\bullet}\l(|\bu|=h,\bM(\bu,\bt)=\bm\r)\nonumber\\
&=\sum_{h\le x\sqrt{\nk}} \sum_{|\bm|=h\atop{\bm\not\in A_{h,a,b}(\kappa)},|\LR(\bm)|\le x\sqrt{\nk}}\frac{|\LR(\bm)|\,h! \,(\nk-h)!}{\nk!(\nk-h)}\prod_{i\geq 1}\binom{n_i}{m_i}{i^{m_i}}.
\end{align}
where, for short, we have written $n_i$ instead of $n_i(\kappa)$. 
{ We now reduce the right-hand side to an expected value with respect to multinomial random variables. Let $(P_i, i\ge 1)$ be multinomial with parameters $h$ and $(in_i/(\nk-1), i\ge 1)$.} Then, for any $\bm=(0, m_1, m_2, \dots)$ such that $|\bm|=h$, we have
$$ `P\l((P_i, i\ge 1)=(m_i, i\ge 1)\r)= \frac{h!}{\prod_{i\ge 1} m_i!} \cdot \prod_{i\ge 1} \l(\frac {in_i}{\nk-1}\r)^{m_i}.$$ 
Now, since $(1-x)^{-1}\leq \exp(2x)$ for $|x|\le 1/2$, we have for all $h\le x\sqrt{\nk}$, and all $\kappa$ large enough,
\begin{align*}
\frac{(\nk-h)!\nk^h}{\nk!}\le \prod_{i=0}^{h-1} \frac1{1-i/{\nk}}\le \prod_{i=0}^{h-1} e^{2i/{\nk}} \le e^{x^2}.
\end{align*}
Note also that, for every $i\ge 1$, we have $n_i!\le n_i^{m_i} (n_i-m_i)!$, so that, rewriting \eqref{eq:content1} in terms of events with respect to $(P_i, i\ge 1)$, we obtain
\begin{align*}
&`P_{\ds(\kappa)}^{\bullet}\l(|\bu|\vee |\LR(\bu,\bt)|\leq x\sqrt{\nk}, \bM(\bu,\bt)\not\in \bigcup_{h\leq x\sqrt{\nk}} A_{h,a,b}(\kappa)\r)\\
& = \sum_{h\le x\sqrt{\nk}} \sum_{|\bm|=h\atop{\bm\not\in A_{h,a,b}(\kappa)}, |\LR(\bm)|\le x \sqrt{\nk}} \frac{|\LR(\bm)|}{\nk-h}\cdot \frac{(\nk-h)! { (\nk-1)^h}}{\nk!} \prod_{i\ge 1} \frac{n_i!}{n_i^{m_i} (n_i-m_i)!} \cdot `P\l((P_i, i\ge 1)=(m_i, i\ge 1)\r)\\
&\le \sum_{h\le x\sqrt{\nk}}\frac {2x}{\sqrt{\nk}} e^{x^2} \sum_{|\bm|=h\atop{\bm\not\in A_{h,a,b}(\kappa)}} `P\l((P_i, i\ge 1)=(m_i, i\ge 1)\r)\\
& \le 2 x^2 e^{x^2} \sup_{h\le x \sqrt{\nk}} `P((P_i, i\ge 1) \not\in A_{h,a,b}(\kappa)).
\end{align*}

Now, we decompose the set of $\bm$ in the right-hand side so as to obtain bad events that are individually simpler to deal with
\begin{align*}
`P_{\ds(\kappa)}^{\bullet}\l(|\bu|\vee |\LR(\bu,\bt)|\le x\sqrt{\nk},\bM(\bu,\bt)\not\in A_{h,a,b}(\kappa)\r)\le 2 x^2 e^{x^2} \sup_{h\le x\sqrt{\nk}}(\zeta_1+\zeta_2)
\end{align*}
where
\begin{align*}
\zeta_1=`P\l(\l|\l(\sum_{i\ge 1} P_i\frac{i-1}2\r)-\frac{h\sigma^2_\kappa}{2}\r|\geq a_\kappa \r)\qquad \text{and}\qquad
\zeta_2= `P\l(\sum_{i\ge 1} i^2 P_i  >b_\kappa\r).
\end{align*}
We now bound the terms $\zeta_1$ and $\zeta_2$ individually.

\medskip
\noindent\textsc{The first term $\zeta_1$.}
Observe first, that 
$$`E\l[\sum_{i\ge 1} P_i \frac{i-1}2\r]=\frac{h\sigma^2_\kappa}{2},$$
so that bounding $\zeta_1$ consists in bounding the deviations of (a function of) a multinomial vector. However, one can write
$$\sum_{i\ge 1} P_i \cdot \frac{i-1}{2} - h \frac{\sigma_\kappa^2}2 \stackrel{d}{=} \sum_{j=1}^h (B_j -`E B_j),$$
where $B_j$, $j=1,\dots, h$, are i.i.d.\ random variables taking value $(i-1)/2$ with probability $ i n_i/(\nk-1)$, for $i\ge 1$. Now, the sums $\sum_{j=1}^\ell (B_i-\Ec{B_j})$, $\ell=0, 1, \dots, h$, form a martingale. We bound their deviations using a concentration inequality from \cite{McDiarmid1998a} (Theorem 3.15), which says that if $S$ is a sum of independent random variable $X_1+\dots+X_n$ such that $`E(S)=\mu$, $\var(S)=V$, and if for all k $X_k-`E(X_k)\leq b$, then $`P(S-\mu\geq t)\leq e^{-t^2/(2V(1+bt/(3V))}$. The variance of $B_j$ may be bounded as follows:
$$\var(B_j) \le \Ec{B_j^2} = \sum_{i\ge 1} \frac{(i-1)^2}4 \frac{i n_i}{ \nk-1} \le \Delta_\kappa \sum_{i\ge 1} \frac{i-1}4 \frac{i n_i}{ \nk-1}{ =} \Delta_\kappa \sigma_\kappa^2/4,$$
for all $\kappa$ large enough.
Now, since $\max\{|B_j-`E(B_j)|: j=0,\dots, h \}\le \Delta_\kappa$, one has, for $h\le x\sqrt{\nk}$,
\begin{align*}
`P\l(\l|\sum_{j=1}^h (B_j-`E B_j)\r|\ge a_\kappa\r)
& \le 2 \exp\l(- \frac{a_\kappa^2}{2 h \Delta_\kappa \sigma^2_\kappa/4 +2\Delta_\kappa a_\kappa/3}\r)\\
& \le 2 \exp\l(- \frac {a_\kappa^2}{ x \sqrt{\nk} \Delta_\kappa \sigma_\kappa^2}\r),
\end{align*}
for all $\kappa$ large enough, since $a_\kappa=\varepsilon_\kappa^{1/4}\sqrt{\nk}=o(\sqrt{\nk})$. It follows that, for every $h\le x\sqrt{\nk}$, we have
\begin{equation}\label{eq:zeta1}
\zeta_1=\sup_{h\leq x\sqrt{\nk}}`P\l(\l|\l(\sum_{i\ge 1} P_i\frac{i-1}2\r)-\frac{h\sigma^2_\kappa}{2}\r|\geq a_\kappa\r) \le 2 \exp\l(- \frac {\varepsilon_\kappa^{-1/2}}{ x  \sigma_\kappa^2}\r).
\end{equation}

\medskip
\noindent\textsc{The second term $\zeta_2$}.\
We bound $\zeta_2$ using the idea we used when bounding $\zeta_1$: one can express the event in terms of independent random variables $B_j$, $j=1,\dots, h$, where $B_j$ takes value $i^2$ with probability $ i n_i/(\nk-1)$.
Observe first that
$$\E{\sum_{i\ge 1} i^2 P_i }=\E{\sum_{j=1}^h B_j} = h \sum_{i\ge 1} i^2 \cdot \frac{i n_i}{ \nk-1} \le h \Delta_\kappa (\sigma_\kappa^2+1).$$
So, we have
\begin{align*}
`P\l(\sum_{i\ge 1} i^2P_i >b_\kappa\r)
&=`P\l(\sum_{j=1}^h B_j > b_\kappa\r)\\
&\le `P\l(\sum_{j=1}^h (B_j - \Ec{B_j})> \frac{b_\kappa}2\r),
\end{align*}
for all $\kappa$ large enough, since $h \Delta_\kappa \le x \varepsilon_\kappa \nk=o(\varepsilon_\kappa^{1/2} \nk)=o(b_\kappa)$. The right-hand side above can be bounded using the martingale inequality in \cite{McDiarmid1998a} (Theorem~3.15). We note that the variance of $B_j$ satisfies
$$\var(B_j)\le \Ec{B_j^2}=\sum_{i\ge 1} i^4 \cdot \frac {in_i}{ \nk-1} \le \Delta_\kappa^3 (\sigma_\kappa^2 +1).$$ 
Since $\max\{|B_i|: i=1,\dots, h \}\le \Delta_\kappa^2$, it follows by McDiarmid's inequality that
\begin{align}\label{eq:zeta3}
\zeta_2\le`P\l(\sum_{j=1}^h (B_j - \Ec{B_j})> \frac{b_\kappa}2\r) 
& \le \exp\l(-\frac{b_\kappa^2/4}{2 x\sqrt{\nk} \Delta_\kappa^3 (\sigma_\kappa^2+1)+2\Delta_\kappa^2 b_\kappa/3}\r)\nonumber\\
& \le \exp\l(-\frac{b_\kappa}{2(x(\sigma_\kappa^2+1)+1/3) \Delta_\kappa^2}\r)\nonumber\\
& = \exp\l(-\frac {\varepsilon_\kappa^{-3/2}}{2x(\sigma^2_\kappa+1)+2/3}\r),
\end{align}
for all $\kappa$ large enough, since $ \Delta_\kappa \sqrt{\nk} =o(b_\kappa)$. 

To complete the proof, it suffices to combine the bounds in \eqref{eq:zeta1}--\eqref{eq:zeta3}, and observe that they imply the claim for $\kappa$ large enough, since the { upper bound in \eqref{eq:zeta3} is much smaller than the one in \eqref{eq:zeta1}.}

\subsection{The structure of a branch with typical content: Proof of Lemma~\ref{lem:discrepancy}}\label{sec:discrepancy}

Finally, we consider the probability that the structure of a branch is not what one expects, in spite of the length and content being close to the typical values. The left hand side in \eref{eq:devi} is bounded by 
\[\sup_{h\leq x\sqrt{\nk}\atop{\bm \in A_{h,a,b}(\kappa)}}\!\!\!\!`P_{\ds(\kappa)}^{\bullet}\l(\l||\R(\bu,\bt)|-\frac{\sigma_\kappa^2}{2}|\bu|\r|\geq c_\kappa ~\Bigg|~ \bM(\bu,\bt)=\bm\r)
= \sup_{h\leq x\sqrt{\nk}\atop{\bm \in A_{h,a,b}(\kappa)}}\!\!\!\!`P\l(\l|\frac{\sigma_\kappa^2}{2}h-\sum_{j\geq 1}\sum_{k=1}^{m_j} U^{(k)}_j\r|\geq c_\kappa \r),
\]
by Proposition \ref{pro:fond-comb} (3), where $U_j^{(k)}$ are independent random variables with $U_j^{(k)}$ uniform on $\{0,1,\dots, j-1\}$. 
By the triangle inequality, the quantity in the right-hand side above is at most
\begin{equation} \label{eq:11-24}
\sup_{h\leq x\sqrt{\nk}\atop{\bm \in A_{h,a,b}(\kappa)}}
 `P\l(\l|\sum_{j\ge 1} m_j\frac{j-1}2-\sum_{j\geq 1}\sum_{k=1}^{m_j} U^{(k)}(j)\r|\geq c_\kappa-\l|\frac{\sigma_\kappa^2h}{2}-\sum_{j\ge 1} m_j\frac{j-1}2\r|\r).
\end{equation}
By definition of $A_{h,a,b}(\kappa)$, and since $c_\kappa>2 a_\kappa$ for all $\kappa$ large enough,
the quantity in \eref{eq:11-24} is bounded by 
\[
\sup_{h\leq x\sqrt{\nk} \atop{\bm \in A_{h,a,b}(\kappa)}}
 `P\l(\l|\sum_{j\ge 1} m_j\frac{j-1}2-\sum_{j\geq 1}\sum_{k=1}^{m_j} U^{(k)}_j\r|\geq \frac{c_\kappa}2\r).
\]

Now, since all the random variables $U_j^{(k)}$, $j\ge 1$, $k=1,\dots, m_j$ are symmetric about their respective mean $(j-1)/2$, one obtains using Chernoff's bounding method
\begin{align}\label{eq:bigbound}
`P\l(\l|\sum_{j\ge 1} m_j\frac{j-1}2-\sum_{j\geq 1}\sum_{k=1}^{m_j} U^{(k)}_j\r|\geq \frac{c_\kappa}2\r)\nonumber
&\leq2\inf_{t\geq 0}e^{-tc_\kappa/2}\E{e^{t\sum_{j\geq 1}\sum_{k=1}^{m_j}\l(U^{(k)}_j-\frac{(j-1)}{2}\r)}}\\
&=2\inf_{t\geq 0}e^{-tc_\kappa/2}\prod_{j\geq 1}\l(\frac{\sinh(tj/2)}{j\sinh(t/2)}\r)^{m_j}\nonumber\\
&\le 2\inf_{t\geq 0}\exp\l(-t\frac{c_\kappa}2+\sum_{j\geq 1}m_j\l(\frac{j^2t^2}{24}-\frac{t^2}{24}+\frac{t^4}{2880}\r)\r)\nonumber\\
&\le 2\inf_{t\in (0,1)}\exp\l(-t\frac{c_\kappa}2+\sum_{j\geq 1}m_j\l(\frac{j^2t^2}{24}-\frac{t^2}{48}\r)\r)\\
&\le 2\inf_{t\in (0,1)}\exp\l(-t\frac{c_\kappa}2+\frac{t^2}{24}\sum_{j\geq 1}m_j j^2\r).\nonumber
\end{align}
Here the third line follows from the bounds $\log(\sinh(s))\leq \log(s)+s^2/6$ and $\log(\sinh(s))\geq \log(s)+s^2/6-s^4/180$ valid for $s\geq 0$. Finally, we obtain
\begin{align*}
\sup_{h\leq x\sqrt{\nk} \atop{\bm \in A_{h,a,b}(\kappa)}}
	 `P\l(\l|\sum_{j\ge 1} m_j\frac{j-1}2-\sum_{j\geq 1}\sum_{k=1}^{m_j} U^{(k)}_j\r|\geq \frac{c_\kappa}2\r)
&\le 2\inf_{t\in(0,1)}\exp\l(-t\frac{c_\kappa}2+ \frac{t^2 b_\kappa}{24}\r)\\
&\le 2 e^{-3c_\kappa^2/(2b_\kappa)},
\end{align*}
upon choosing $t=6c_\kappa/b_\kappa$, which is indeed in $(0,1)$ for $\kappa$ large enough (we restricted the range of $t$ in \eqref{eq:bigbound}). This completes the proof since $3c_\kappa^2/(2b_\kappa)=3\varepsilon_\kappa^{-3/4}/2\ge \varepsilon_\kappa^{-1/2}$, for all $\kappa$ large enough.

\section{The limit of rescaled Galton--Watson trees: Proof of Proposition~\ref{pro:Ald}}
\label{sec:pt}

{ Consider the family tree of a Galton-Watson tree $\bt$ with offspring distribution $\mu=(\mu_i,i\geq 0)$ starting with one individual. Let $`P_\mu$ be the probability distribution of $\bt$.}
Denote by  $\wh{\ds}_\bt:=(\wh{n}_i(\bt),i\geq 0)$ the empirical degree sequence of $\bt$, let {
\begin{align*}
\wh{\mu}_i &= \wh{n}_i(\bt)/|\bt|,\\
\wh{\sigma}^2&=\sum_{i\ge 0} i^2\frac{\wh{n}_{i}(\bt)}{|\bt|-1}-1\\
\wh{\Delta} &=\max\{i:\wh{n}_i>0\}.
\end{align*}
Note that $\wh{\sigma}^2$ is not the variance of the empirical distribution $(\wh{\mu}_i, i\ge 0)$ but has been chosen to be consistent with the definition of $\sigma_{\ds(\kappa)}^2$ in \eqref{eq:sigma}. Write $`P_\mu^n(\,\cdot\,)=`P_\mu(\,\cdot\,|\,|\bt|=n)$.} In what follows, all the assertions containing `` $`P_\mu^n$'' are to be understood ``for $n$ such that $`P_\mu(|\bt|=n)>0$''; similarly, the limit with respect to  $`P_\mu^n$ are to be understood in the same manner, along subsequences included in $\{n: `P_\mu(|\bt|=n)>0\}$.
\begin{lem}\label{lem:GW}
Assume that $\mu$ has mean 1 and variance $\sigma^2_\mu\in(0,+\infty)$. Then 
under  $`P_\mu^n$,
\begin{equation}\label{eq:cv-tr}
(\hat{\mu},\wh{\sigma}^2,{\wh{\Delta}}/{\sqrt{n}})\dd (\mu,\sigma^2_\mu, 0),
\end{equation}
where the convergence holds in the space ${\cal M}(`N)\times `R\times `R$ equipped with the product topology.
\end{lem}
{ In this lemma,} ${\cal M}(`N)$ is the set of probability measures on $`N$. The topology on ${\cal M}(`N)$ is metrizable, for example, by the distance 
\[D(\nu,\nu')=\sum_{i\geq 0} { \frac{1}{2^i}}d_{\textup{TV}}(\nu[i],\nu'[i])\]
where $\nu[i]$ is the distribution of the $i$th first marginals under $\nu$ and $d_{\textup{TV}}$ is the distance in total variation. Since here the limit is the deterministic measure $\mu$, it suffices to show that, for all $i$,  $\hat{\mu}_i\to \mu_i$ in probability as $n\to\infty$. With $D$ it is easy to construct a metric on ${\cal M}(`N)\times `R\times `R$ making of this space a Polish space. Hence, by the Skohorod theorem there exists a probability space where versions of $(\hat{\mu},\wh{\sigma}^2,{\wh{\Delta}}/{\sqrt{n}})$ under $`P_\mu^n$ converges almost surely to $(\mu,\sigma^2_\mu, 0)$. So on the conditional space, the hypotheses of Theorem~\ref{thm:main_gh} hold almost surely, and then its conclusion, which is a limit in distribution, also holds. { Of course, we do not mean that any sequence of trees for which the degree distribution satisfies the conditions of Theorem~\ref{thm:main_gh} converges to the continuum random tree; one also needs that for any fixed $\kappa$, conditional on the degree sequence $\ds$ the trees are distributed according to $`P_\ds$. This fact certainly holds for conditioned Galton--Watson trees: under $`P_\mu$ all trees with the same degree sequence occur with the same probability, and conditional on its degree sequence $\ds$, a Galton--Watson tree is precisely distributed according to $`P_{\ds}$. To summarise, to prove Proposition~\ref{pro:Ald} it suffices to prove Lemma~\ref{lem:GW}.}

\begin{proof}[Proof of Lemma \ref{lem:GW}]
The claim is about properties of the degree sequence of Galton--Watson trees conditioned on their total progeny. We first provide a way to construct the degree sequence. Consider the \LL walk $S_n$ associated with a tree $\bt$ under $`P_\mu^n$; the degree sequence of the tree $\bt$ is essentially (just shift by one) the empirical distribution of the increments of $S_n$. More precisely, consider first a random walk $W=(W_k,k=0,\dots,n)$, with i.i.d.\ increments $X_k=W_k-W_{k-1},k=1,\dots,n$ with distribution \[\nu_i=`P(X_k=i)=\mu_{i+1}\quad i\geq -1;\]
then $S=(S_0,\dots,S_n)$ is distributed as $W$ conditioned on
$W\in A^+_{-1}(n)$ where 
$$A^+_{-1}(n)=\{w=(w_0,\dots,w_n)~:~ w_0=0, w_k\geq 0, 1\le k <n,w_n=-1\}$$ is the set of discrete excursions of length $n$. 

Write $K_i=\#\{ k : X_k=i-1 \}$, and $K=(K_i,i\geq 0)$. Then, if $W\in A^+_{-1}(n)$, the sequence $K=(K_i, i\ge 0)$ is distributed as the degree sequence of a tree under $`P_\mu^n$. In other words, we have
\[`P(K\in B~|~W\in A^+_{-1}(n))=`P_{\mu}^n((\wh{n_i}(\bt),i\geq 0)\in B ).\] 
By the rotation principle, we may remove the positivity condition :
\[`P(K\in B~|~W_n=-1)=`P_{\mu}^n((\wh{n_i}(\bt),i\geq 0)\in B ).\] 

Our aim is now to show that the condition that $W$ is a bridge imposed by $W_n=-1$ does not completely wreck the properties of $W$ in the following sense:  let ${\cal F}_k=\sigma(W_0,\dots,W_k)$ be the $\sigma$-field generated by the $k$ first $W_i$; then there exists a constant $c\in (0,\infty)$ such that for any $n$ large enough, and for any event $B\in {\cal F}_{\floor{n/2}}$ one has
\begin{equation}\label{eq:absolute_cont}`P(B ~|~ W_n=-1) \leq c\, `P(B).\end{equation}
That is: any event $B$ in ${\cal F}_{\floor{n/2}}$ with a very small probability for a standard (unconditioned) random walk also has a small probability in the bridge case (conditional on $W_n=-1$). 
The argument proving this claim is given in \citet{JM}, page 662 and goes as follows: 
\begin{align*}
`P(B\,|\,W_{n}=-1)
&=\sum_{x} `P(B\,|\,W_{\floor{n/2}}=x,W_{n}=-1) \cdot \frac{`P(W_{\floor{n/2}}=x,W_{n}=-1)}{`P(W_n=-1)}\\
&=\sum_{x} `P(B\, |\, W_{\floor{n/2}}=x)`P(W_{\floor{n/2}}=x)\cdot \frac{`P(W_{n-\floor{n/2}}=-x-1)}{`P(W_n=-1)}.
\end{align*}
It then suffices to (a) observe that $\sup_x `P(W_{n-\floor{n/2}}=-x-1)\le c/\sqrt{n}$ for some constant $c_1\in (0,\infty)$ \cite[][Theorem 2.2 p. 76]{PET1995}, and (b) use a local limit theorem to show that $`P(W_n=-1)\ge c_2 \sqrt n$, for some constant $c_2\in (0,\infty)$ and all $n$ large enough  \cite[page 233]{GK-54}. This gives the result in \eqref{eq:absolute_cont} with $c=c_1/c_2$.

Now using that the increments $(X_1,\dots,X_n)$ under $`P(\,\cdot\,| \, W_n=-1)$ are exchangeable, any concentration principle for the first half of them easily extends to the second half (the easy details are omitted). Consider the degree sequence induced by the first half of the walk:
let $K^{1/2}_i=\#\{ k : X_k=i-1, k \leq \floor{n/2} \}$, and note that the $K_i^{1/2}$ are  ${\cal F}_{\floor{n/2}}$-measurable.
For $W$ (that is, with no conditioning), we have
\bq\label{ec:cvsigma_mu}
\frac 1 {\floor{n/2}} \sum_{i\ge 0} K^{1/2}_i(i-1)^2=\frac{1}{\floor{n/2}} \sum_{j=1}^{\floor{n/2}} X_j^2\xrightarrow[n\to\infty]{} \Ec{X_1^2}=\sigma_\mu^2
\eq
by the law of large number, since $X_i$ owns a (finite) moment of order 2. 
Hence, for any  $`e>0$, writing
\[Ev(`e)= \l\{ \l|\frac 1{\floor{n/2}}\sum_{i\ge 0} K^{1/2}_i(i-1)^2-\sigma_\mu^2 \r|\geq `e \r\},\]
we have
$`P(Ev(`e))\to 0$ and thus, according to the bound in \eqref{eq:absolute_cont}, $`P(Ev(`e) |W_n=-1)\to 0$, as $n\to\infty$. Using the argument twice (one for each half of the walk) yields convergence $\wh{\sigma}^2\to \sigma_\mu^2$ in probability as $n\to\infty$.

The same argument also proves that 
$$`P\l( \l|\frac{K^{1/2}_i}{\floor{n/2}}-\mu_i\r|\geq `e~\Bigg|\,W_n=-1\r)\to 0,$$ 
which yields $\wh{\mu_i}\to \mu_i$ in probability.

The fact that $\wh{\Delta}=o(\sqrt{n})$ (in probability) under $`P_\mu^n$ is also a consequence of the convergence of the sum given in \eref{ec:cvsigma_mu}. To see this, let $C(`a)= \{k ~:~`P(X_1^2\geq k)\geq `a/k\}$. Since $\Ec{X_1^2}=\sum_{k\ge 0} `P(X_1^2\geq k)<+\infty$, then $k`P(X_1^2\geq k)\to 0$ , entailing  $\#C(`a)<+\infty$ for any $\alpha>0$. In particular, for any $`e>0$, 
$$\#\{n : n`P(X_1^2\geq `e n)\geq `a/`e\}<+\infty.$$ Taking $`a=`e`e'$, one obtains that $\#\{n : n`P(X_1^2\geq `e n)\geq `e'\}<+\infty$, which implies that
$$`P(\max \{X_i:i\le n/2\}\geq `e \sqrt n)\le n `P(X_1^2\ge `e n)\xrightarrow [n\to\infty]{} 0.$$ So under the unconditioned law one has $\wh\Delta = o(\sqrt n)$; we complete the proof using the bound in \eqref{eq:absolute_cont}.
\end{proof}

\section{Application to constrained coalescing processes}\label{sec:coagulation}

In this final section, we discuss an application of Theorem~\ref{thm:main_gh} to a coalescence process with particles having constrained valences.

The famous additive coalescent \cite{AlPi1998a,Bertoin2000a, Bertoin2006, Pitman1999b,Pitman2006} can be seen as arising from the following natural microscopic description. Consider a set of $n$ distinct particles $\{1,2,\dots, n\}$. The particles are initially free, and form $n$ clusters; the clusters are organised as rooted trees. The clusters merge according to the following dynamics. At each step, choose a particle $u$ uniformly at random; it belongs to some cluster $T$ rooted at $r$. Choose uniformly a second cluster $T'\neq T$, with root $r'$. Add an edge between $r'$ and $u$ to obtain a new cluster rooted at $r$. At each step, the system consists of a forest of general rooted labelled trees (an acyclic graph on $\{1,2,\dots, n\}$ with a distinguished node per connected component). The process stops after $n-1$ steps, when the system consists of a single rooted labelled tree. The final tree is then uniform among all rooted labelled trees. 

One can similarly define a system of coalescing particles where the degrees would be constrained. Different algorithms might be used, depending on the precise way the uniform choices are made, that yield \emph{a priori} different trees.

\medskip
\noindent\textsc{Labelled particles.} Consider the set of particles $\{1,2,\dots, n\}$, and a set of degrees $c_1\le c_2\le\dots\le c_n$. Write $\ds=(n_i, i\ge 0)$ for the associated degree sequence, $n_i=\#\{j: c_j=i\}$. Assign randomly the particles a degree. For instance, this can be done using a random permutation $\sigma=(\sigma(1),\dots,\sigma(n))$ of $\{1,2,\dots, n\}$ and assigning degree $c_{\sigma(i)}$ to particle $i$. Think now of the particle $i$ as initially having edges to $c_{\sigma(i)}$ free slots that can each contain a single particle. The particles will now merge to form clusters. Each cluster is represented by a tree with a distinguished vertex (the root). Initially, each particle sits in a tree containing a single node (which is then also the root). Proceed with the following algorithm to merge the particles, as long as there are free slots left:
\begin{itemize}
	\item Pick a free slot $s$ uniformly at random; say it is bound to particle $p$ lying in the cluster rooted at $r$.
	\item Pick \emph{another} cluster, uniformly at random, rooted at some node $r'$.
	\item Merge the two clusters by assigning $r'$ to the free slot $s$; this creates an edge between the particles $p$ and $r'$, and removes the slot $s$ from the set of free slots. The new cluster is rooted at $r$.
\end{itemize}
At every iteration, precisely one slot is filled and the process stops after $n-1$ steps. The process yields a random tree \emph{labelled} tree $T_n^L$. 

The labelled tree $T_n^L$ is uniform in the set of labelled trees having the same specified degree sequence. To see this, just consider the encoding of the process by the final labelled tree, together with a labelling of the edge indicating their order of appearance. At iteration $i\in \{1,\dots, n-1\}$, there are $n-i$ free slots left and $n-i+1$ connected components, so that the probability that any couple free slot/other connected component is precisely 
$$\frac 1 {(n-i)^2}.$$
Overall, the probability to obtain any particular pairing free slots/particles together with a history is
$$\prod_{i=1}^n \frac 1{(n-i)^2}=\frac 1 {(n-1)!^2}.$$

The same particle adjacency ---hence the same labelled tree--- is obtained by the $\prod_{j=1}^n c_j!$ ways to pair the free slots with particles; and for any labelled tree there are exactly $(n-1)!$ distinct histories. Finally, among the $n!$ ways to assign the labels to particles in the first place, $\prod_{i\ge 0}n_i!$ correspond to the degree/label pattern of the tree, it follows that the probability of seeing any labelled tree after $n-1$ iterations is precisely
\begin{equation}\label{eq:coalescing_labelled}
\frac{\prod_{i\ge 0}n_i!}{n!} \times \frac 1 {(n-1)!^2} \times (n-1)! \times \prod_{i=1}^n c_i! = \frac{\prod_{i\ge 0} i!^{n_i}}{(n-1)!} \times \binom{n}{(n_i, i\ge 0)}^{-1},
\end{equation}
which depends only on the degree sequence, so that trees with the same degree sequence are chosen uniformly. (This is also, as it should, the inverse of the number of labelled trees with degree sequence given by $\ds=(n_i, i\ge 0)$ \cite[][Example 6.2.2]{Pitman2006}.)

\medskip
\noindent\textsc{Unlabelled particles.} Consider a degree sequence in the form of $\ds=(n_i, i\ge 0)$ where $n_i$ denotes the number of nodes of degree $i$. For $c_1\le c_2\le \dots \le c_n$ of size $n$. So $\sum_{i\ge 0} c_i=n-1$. As before, we think of the particles as having empty slots, but since there are no labels, we impose that the slots of any given particle be ordered. The particles then merge according to the same algorithm, in order to distinguish particles use the canonical labelling giving label $i$ to the particle with degree $c_i$. After forgetting the canonical labelling, the process yields a plane tree $T_n$. 

Again, the plane tree $T_n$ is uniform among all plane trees with the correct degree sequence. The arguments are similar, only simpler, to those we used in the labelled case. Since, for a given plane tree, there are $\prod_{i\ge 0} n_i!$ ways to assign the canonical labels to the nodes, the probability to obtain any given plane tree is 
$$\prod_{i\ge 0} n_i! \times \frac 1 {(n-1)!^2} \times (n-1)! = n \binom{n}{(n_i, i\ge 0)}^{-1}$$

\medskip
In these coalescing particle systems, one of the parameters of interest is the metric structure of the cluster (structure of the ``molecule'') eventually obtained after all particles have coalesced into a single component. In the unrestricted case, the metric structure is described by the CRT of Aldous. Our result shows that the quenched version, conditional on the degree sequence, is also valid under reasonable conditions on the degree sequence imposed. Results for Galton--Watson trees conditioned on the size only are recovered by sampling the degree sequence. 

For instance, to recover the unrestricted version of the merging process, one can sample $n$ independent Poisson(1) random variables, and keep them if their sum equals $n-1$; the $n$ exchangeable values obtained are then the degrees $C_1,C_2,\dots, C_n$ of the $n$ particles.

\subsection*{Acknowledgements }
{  We are grateful to the referees for the many relevant remarks they made on the paper.}
{
\setlength{\bibsep}{.2em}
\small
\bibliographystyle{abbrvnat}
\bibliography{bibfile,bib_nic}
}

\end{document}